\documentclass[12pt]{amsart}
\usepackage{a4wide}
\usepackage{enumerate}
\usepackage{graphicx}
\usepackage{color}
\usepackage{amsmath}
\usepackage{enumitem}
\allowdisplaybreaks

\let\pa\partial  
\let\na\nabla  
\let\eps\varepsilon  
\newcommand{\N}{{\mathbb N}}  
\newcommand{\R}{{\mathbb R}} 
\newcommand{\diver}{\operatorname{div}}

\newtheorem{theorem}{Theorem}   
\newtheorem{lemma}[theorem]{Lemma}   
   
\newtheorem{remark}[theorem]{Remark}   
\newtheorem{corollary}[theorem]{Corollary}  
\newtheorem{definition}{Definition}

\begin{document}  

\title[Reaction-cross-diffusion systems]{Global renormalized
solutions to reaction-cross-diffusion systems}
 
\author[X. Chen]{Xiuqing Chen}
\address{School of Sciences, Beijing University of Posts and Telecommunications,
Beijing 100876, China}
\email{buptxchen@yahoo.com}

\author{Ansgar J\"ungel}
\address{Institute for Analysis and Scientific Computing, Vienna University of  
Technology, Wiedner Hauptstra\ss e 8--10, 1040 Wien, Austria}
\email{juengel@tuwien.ac.at} 

\date{\today}

\thanks{The first author acknowledges support from the National Natural Science
Foundation of China (NSFC), grant 11471050, and from the China Scholarship Council
(CSC), file no.\ 201706475001, who financed his stay in Vienna.
The second author acknowledges partial support from
the Austrian Science Fund (FWF), grants P27352, P30000, F65, and W1245} 

\begin{abstract}
The global-in-time existence of renormalized solutions to 
reaction-cross-diffu\-sion systems for an arbitrary number of variables
in bounded domains with no-flux boundary conditions is proved. 
The cross-diffusion part describes the segregation of population species
and is a generalization of the Shigesada-Kawasaki-Teramoto model. 
The diffusion matrix is not diagonal and generally neither symmetric nor
positive semi-definite, but the system possesses a formal gradient-flow or 
entropy structure. The reaction part includes reversible reactions of 
mass-action kinetics and does not obey any growth condition. The existence result
generalizes both the condition on the reaction part required in the boundedness-by-entropy
method and the proof of J.~Fischer for reaction-diffusion
systems with diagonal diffusion matrices.
\end{abstract}

% \paragraph{Keywords:}  
 
% \paragraph{AMS classification:}  

\maketitle

%%%%%%%%%%%%%%%%%%%%%%%%%%%%%%%%%%%%%%%%%%%%%%%%%%%%%%%%%%%%%%%%%%%%%%%%%%%%%%%

\section{Introduction}

Multi-species systems from thermodynamics, population dynamics, and cell biology,
for instance, are often modeled by reaction-cross-diffusion equations.
Cross diffusion occurs when the gradient of the density of one species induces a
flux of another species. Therefore, cross-diffusion systems are strongly coupled, and
only weak solutions can be expected. When the reaction terms grow too fast
with the densities, there is no control of these terms and the definition
of a weak or distributional solution is generally impossible. For this reason,
growth restrictions have been imposed
on the reactions in the literature \cite{Jue15,LeMo17}. Roughly speaking,
the reaction terms cannot grow faster than linear. The
cross-diffusion systems from physics, biology, and chemistry often allow for entropy
(or free energy) estimates that prevent the global blowup of solutions,
but the bounds are not sufficient to define weak solutions. This suggests the
concept of {\em renormalized} solutions. This approach was successfully realized
by J.~Fischer \cite{Fis15} for reaction-diffusion systems, i.e.\ weakly coupled
equations. In this paper, we extend his approach to strongly coupled systems.

More specifically, we investigate cross-diffusion systems from population dynamics,
which extend the well-known model of Shigesada, Kawasaki, and Teramoto \cite{SKT79}.
The evolution of the density $u_i=u_i(x,t)$ of the $i$th population species
is governed by the equation
\begin{equation}\label{1.eq}
  \pa_t u_i - \diver\bigg(\sum_{j=1}^n A_{ij}(u)\na u_j - u_ib_i\bigg) = f_i(u)
	\quad\mbox{in }\Omega,\ i=1,\ldots,n,
\end{equation}
where $A_{ij}(u)$ are the density-dependent diffusion coefficients, 
$u=(u_1,\ldots,u_n)$ is the density vector, $b_i=(b_{i1},\ldots,b_{in})$ is
a given vector which describes the environmental potential acting on the $i$th species,
$f_i(u)$ is a reaction term describing the population growth dynamics,
$\Omega\subset\R^n$ ($d\ge 1$) is a bounded domain, and $n\in\N$ is the number of 
species. We impose no-flux boundary and initial conditions,
\begin{equation}\label{1.bic}
  \bigg(\sum_{j=1}^n A_{ij}(u)\na u_j - u_ib_i\bigg)\cdot\nu = 0\quad\mbox{on }\pa\Omega,
	\quad u_i(\cdot,0)=u_i^0\quad\mbox{in }\Omega,\ i=1,\ldots,n,
\end{equation}
where $\nu$ is the exterior unit normal vector on $\pa\Omega$.
The diffusion coefficients are given by
\begin{equation}\label{1.A}
  A_{ij}(u) = \delta_{ij}\bigg(a_{i0} + \sum_{k=1}^n a_{ik}u_k\bigg) + a_{ij}u_i,
	\quad i,j=1,\ldots,n,
\end{equation}
where $a_{i0}\ge 0$, $a_{ij}\ge 0$ for $i,j=1,\ldots,n$, and $\delta_{ij}$ is the
Kronecker delta. The reaction terms are often given by Lotka-Volterra-type
expressions, but we allow for fast growing populations, and no growth condition
on $f_i$ will be imposed.
Observe that \eqref{1.eq} can be written in a more compact form as
\begin{equation}\label{1.comp}
  \pa_t u - \diver(A(u)\na u-ub) = f(u),
\end{equation}
where the matrix $ub$ is defined as $(ub)_{ij}:=u_ib_{ij}$.

For $n=2$, we recover the population model of Shigesada, Kawasaki, and Teramoto
\cite{SKT79}, which descibes the segregation of two population species. 
Equations \eqref{1.eq}, \eqref{1.A} for an arbitrary number of species, $n\ge 2$, 
have been formally
derived in \cite{ZaJu17} from a random-walk on a lattice in the diffusion limit.

For the analysis, we impose two key assumptions. First, we assume that the
reaction terms $f_i$ are continuous on $[0,\infty)^n$ and that there are numbers
$\pi_1,\ldots,\pi_n>0$ and $\lambda_1,\ldots,\lambda_n\in\R$ 
such that for all $u\in(0,\infty)^n$,
\begin{equation}\label{1.f}
  \sum_{i=1}^n\pi_i f_i(u)(\log u_i+\lambda_i)\le 0.
\end{equation}
These conditions imply the quasi-positivity property
$f_i(u)\ge 0$ for all $u\in[0,\infty)^n$ with $u_i=0$, which is a necessary condition 
for having nonnegative solutions to \eqref{1.eq}.
Moreover, it ensures that the so-called entropy density
\begin{equation}\label{1.h}
  h(u) = \sum_{i=1}^n\pi_i h_i(u_i), 
	\quad h_i(s)=s(\log s-1)+\lambda_i)+e^{-\lambda_i},
\end{equation}
is a Lyapunov functional for the reaction system $\pa_t u_i = f_i(u)$ if $\pi_i=1$
for all $i$. Condition \eqref{1.f} (with $\pi_i=1$) was also used in \cite{Fis15}. 
Compared to \cite{Fis15}, we do not assume
local Lipschitz continuity of $f_i$ but only continuity.

Second, to ensure that the entropy \eqref{1.h} yields a 
Lyapunov functional also for the full system
\eqref{1.eq}, we need to impose some conditions on the coefficients $a_{ij}$.
It was shown in \cite{CDJ17} that a sufficient requirement is either
the weak cross-diffusion assumption
\begin{equation}\label{1.wcd}
  \alpha := \min_{i=1,\ldots,n}\bigg(a_{ii} - \frac14\sum_{j=1}^n
	\big(\sqrt{a_{ij}}-\sqrt{a_{ji}}\big)^2\bigg) > 0, 
\end{equation}
or the detailed-balance condition
\begin{equation}\label{1.dbc}
  \pi_i a_{ij}=\pi_j a_{ji}>0 \quad\mbox{for all }i,j=1,\ldots,n,\ i\neq j.
\end{equation}
In the former case, we may choose $\pi_i=1$ in \eqref{1.f} and \eqref{1.h}.
Condition \eqref{1.wcd} requires implicitly that $a_{ii}>0$, while \eqref{1.dbc}
requires that $a_{ij}>0$ for $i\neq j$. A formal computation shows that 
$$
  \frac{d}{dt}\int_\Omega h(u)dx 
	+ \int_\Omega\na u:\big(h''(u)A(u)\na u - h''(u)ub\big) dx
	\le 0,
$$
where "`:"' is the Frobenuis matrix product and $h''(u)$ is the Hessian of the 
entropy density $h(u)$. The drift term $\na u:h''(u)ub=\na u:b$ can be estimated 
by using the Cauchy-Schwarz inequality, and the matrix $h''(u)A(u)$ is 
positive definite for $u_i>0$. More precisely, assuming \eqref{1.f} and \eqref{1.wcd},
it follows that
\begin{equation}\label{1.epi1}
  \frac{d}{dt}\int_\Omega h(u)dx + 4\int_\Omega\sum_{i=1}^n a_{i0}|\na\sqrt{u_i}|^2dx
	+ \alpha\int_\Omega\sum_{i=1}^n|\na u_i|^2 dx \le 0,
\end{equation}
while under the conditions \eqref{1.f} and \eqref{1.dbc}, we have
\begin{equation}\label{1.epi2}
  \frac{d}{dt}\int_\Omega h(u)dx + 2\int_\Omega\sum_{i=1}^n\pi_i\bigg(
	2a_{i0}|\na \sqrt{u_i}|^2 + a_{ii}|\na u_i|^2
	+ \sum_{j\neq i}a_{ij}|\na\sqrt{u_iu_j}|^2\bigg)dx \le 0,
\end{equation}
thus obtaining gradient estimates for $\sqrt{u_i}$ if $a_{i0}>0$ and for
$u_i$ if $\alpha>0$ or $a_{ii}>0$. 

Let us briefly comment on conditions \eqref{1.wcd} and \eqref{1.dbc}; for details, we
refer to \cite{CDJ17}. 
If $(a_{ij})$ is symmetric and $a_{ii}>0$, then \eqref{1.wcd} is fulfilled. 
Otherwise, the condition requires that 
the coefficient $a_{ii}$ is larger than the
``defect of symmetry'' of the matrix $(a_{ij})$
or that the cross-diffusion coefficients
$a_{ij}$ are small compared to the self-diffusion coefficients $a_{ii}$. 
Assumption \eqref{1.dbc} is the detailed-balance
condition for the Markov chain associated to $(a_{ij})$, and $(\pi_1,\ldots,\pi_n)$
is the reversible measure of the Markov chain. It turns out that
this condition is equivalent to the symmetry of $h''(u)A(u)$ or, equivalently, 
of the so-called Onsager matrix $B=A(u)h''(u)^{-1}$. 
This indicates a close relationship between symmetry and
reversibility, which is well known in nonequilibrim thermodynamics
(also see \cite[Section 4.3]{Jue16}).

Before stating our main result, we review the state of the art.
The existence of global weak solutions to \eqref{1.eq} was proved
in \cite{ChJu04,ChJu06} for two species and in \cite{CDJ17} for an arbitrary
number of species. The global existence was also proved for diffusion matrices
with nonlinear coefficients $A_{ij}(u)$. 
The case of sublinearly growing coefficients was
treated in \cite{DLM14}, whereas superlinear growth was analyzed in 
\cite{DLMT15,Jue15} (for two species). The results were generalized to
$n$ species in \cite{CDJ17,LeMo17}. In \cite{Jue15,LeMo17}, the condition on
the reaction terms is as follows (formulated here for the linear case): 
There exists $C>0$ such that for all $u\in(0,\infty)^n$,
\begin{equation}\label{1.f2}
  \sum_{i=1}^n f_i(u)\log u_i\le C(1+h(u)).
\end{equation}
This inequality is satisfied for functions $f_i$ which grow at most linearly.
Comparing this condition with \eqref{1.f}, written as
$\sum_{i=1}^n f_i(u)\log u_i\le -\sum_{i=1}^n\lambda_i f_i(u)$, we see that
for $\lambda_i<0$, this inequality is usually weaker than \eqref{1.f2}.

When the diffusion matrix $A(u)$ is diagonal and constant, global existence
results for \eqref{1.eq} were shown in \cite{DFPV07} and later extended to
$L^1$ data in \cite{PiRo16}.
A more general result, assuming space-time dependent $A_{ij}$ and mass
action kinetics, was shown in \cite{Kra11}. 
In \cite{DeFe15}, strongly degenerate diffusion systems, still with diagonal
diffusion matrices, were analyzed.  
When the reactions have quadratic growth and are 
dissipative in the sense $\sum_{i=1}^n f_i(u)\le 0$,
classical solutions can be obtained \cite{PSY16}. 
If the diffusion coefficients are close to each other, even superquadratic
growth in the reaction terms is possible \cite{FLS16}. On the other hand,
it was shown in \cite{PiSc97} that the $L^\infty$ norm of the solutions to
\eqref{1.eq} with density-dependent diffusion coefficients may blow up
in finite time, even if the total mass is controlled. As mentioned above,
existence of renormalized solutions for general reaction terms,
involving a single reversible reaction with mass-action kinetics, was proved
in \cite{Fis15}. Furthermore, it was shown that the renormalized
solutions satisfy a weak entropy-production inequality \cite{Fis17}
and that they converge exponentially fast in the $L^1$ norm to the equilibrium
\cite{FeTa17}.

In this paper, we combine the entropy method used in \cite{CDJ17,DLMT15,LeMo17}
and the concept of renormalized solutions of \cite{Fis15}.
Our hypotheses are as follows. 
\renewcommand{\labelenumi}{(H\theenumi)}
\begin{enumerate}[leftmargin=10mm,itemsep=5pt]
\item Drift term: $b=(b_1,\ldots,b_n)$, $b_i\in L^\infty(0,T;L^\infty(\Omega;\R^n))$,
$i=1,\ldots,n$.
\item Reaction terms: $f=(f_1,\ldots,f_n)\in C^0([0,\infty)^n;\R^n)$.
\item Initial data: $u^0=(u_1^0,\ldots,u_n^0)$ is measurable, $u_i^0\ge 0$ in
$\Omega$, $i=1,\ldots,n$, and $\int_\Omega h(u^0)dx<\infty$, where $h$ is
defined in \eqref{1.h}.
\end{enumerate}

\begin{description}[leftmargin=10mm]
\item[\rm (H4)] There exist numbers $\pi_i>0$ and $\lambda_i\in\R$, $i=1,\ldots,n$, 
such that for all $u=(u_1,\ldots,u_n)\in(0,\infty)^n$, inequality \eqref{1.f} holds.

\item[\rm (H5')] The weak cross-diffusion condition \eqref{1.wcd}
holds and $\pi_i=1$ for all $i=1,\ldots,n$.

\item[\rm (H5'')] The detailed-balance condition \eqref{1.dbc} and either
$a_{i0}>0$ for all $i=1,\ldots,n$ or $a_{ii}>0$ for all $i=1,\ldots,n$ hold.
\end{description}

\begin{definition}[Renormalized solution]
We call $u=(u_1,\ldots,u_n)$ a {\em renormalized solution} to 
\eqref{1.eq}-\eqref{1.bic} if for all $T>0$, $u_i\in L^2(0,T;H^1(\Omega))$
or $\sqrt{u_i}\in L^2(0,T;H^1(\Omega))$,
and for any $\xi\in C^\infty([0,\infty)^n)$ satisfying 
$\xi'\in C_0^\infty([0,\infty)^n;\R^n)$ and  
$\phi\in C_0^\infty(\overline\Omega\times[0,T))$, it holds that
\begin{align}
  -\int_0^T\int_\Omega & \xi(u)\pa_t\phi dxdt - \int_\Omega\xi(u^0)\phi(\cdot,0)dx 
	\nonumber \\
	&= -\sum_{i,k=1}^n\int_0^T\int_\Omega\pa_i\pa_k\xi(u)
	\bigg(\sum_{j=1}^n A_{ij}(u)\na u_j-u_ib_i\bigg)\cdot\na u_k\phi dxdt 
	\label{1.renorm} \\
	&\phantom{xx}{}- \sum_{i=1}^n\int_0^T\int_\Omega\pa_i\xi(u)
  \bigg(\sum_{j=1}^n A_{ij}(u)\na u_j-u_ib_i\bigg)\cdot\na\phi dxdt \nonumber \\
  &\phantom{xx}{}+ \sum_{i=1}^n\int_0^T\int_\Omega\pa_i\xi(u)f_i(u)\phi dxdt. \nonumber
\end{align}
\end{definition}

In the definition, $\xi'$ is the gradient of $\xi$ and $\pa_i\xi(u)=\pa\xi/\pa x_i$ 
the $i$th partial derivative. Note that $\xi'$ is assumed to have compact support,
so all integrals are well defined. 
If $\sqrt{u_i}\in L^2(0,T;H^1(\Omega))$, the expression $A_{ij}(u)\na u_j$ is
interpreted as $2A_{ij}(u)\sqrt{u_j}\na\sqrt{u_j}$.

The main result reads as follows.

\begin{theorem}[Global existence]\label{thm.ex}
Let (H1)-(H4) and either (H5') or (H5'') hold. Then there exists a renormalized
solution $u=(u_1,\ldots,u_n)$ to \eqref{1.eq}-\eqref{1.A}
satisfying $u_i\ge 0$ in $\Omega$ and $\int_\Omega h(u(t))dx<\infty$
for all $t>0$.
\end{theorem}

The proof is based on the entropy method of \cite{Jue15} and the
approximation scheme of \cite{Fis15}. In the following, we sketch the key ideas.

{\em Step 1: Approximation scheme.} Introducing the entropy variable
$w=(w_1,\ldots,w_n)$ with $w_i=\pa h/\pa u_i=\log u_i+\lambda_i$, equations
\eqref{1.eq} can be equivalently written as
\begin{equation}\label{1.B}
  \pa_t u(w) - \diver(B\na w-u(w)b) = f(u(w)),
\end{equation}
where $B=A(u)h''(u)^{-1}$ is positive definite
if condition \eqref{1.wcd} or \eqref{1.dbc} is assumed, and 
$u(w)=(h')^{-1}(w)$ with components $u_i(w)=\exp(w_i-\lambda_i)$ is interpreted
as a function of the entropy variable $w$.
System \eqref{1.B} is approximated by an implicit Euler scheme with time step 
$\tau>0$, an elliptic regularization of the type $\eps((-\Delta)^m w+w)$
for some $\eps>0$ and $m\in\N$, and a regularized reaction term
$$
  f_i^\delta(u) = \frac{f_i(u)}{1+\delta|f(u)|}
$$
with parameter $\delta>0$. The Euler scheme avoids
issues with the (low) time regularity; the elliptic regularization yields the
regularity $w\in H^m(\Omega)\hookrightarrow L^\infty(\Omega)$ if $m>d/2$;
and the regularized reaction term is bounded, which allows us to apply the
entropy method of \cite{Jue15}. 

{\em Step 2: Limit $(\eps,\tau)\to 0$.}
A discrete version of the entropy-production inequality \eqref{1.epi1} 
or \eqref{1.epi2} yields
estimates uniform in $\tau$ and $\eps$ but not in $\delta$. Using the Aubin-Lions
lemma, we can pass to the limit $\eps=\tau\to 0$, and we infer the strong
convergence $u^{(\tau)}\to u$ in $L^1(\Omega\times(0,T))$. In order to pass
to the limit in equations \eqref{1.eq} with the right-hand side replaced by
$f_i^\delta(u)$, we need to distinguish between the
cases $a_{ii}>0$, which yields uniform estimates for $\na u_i^{(\tau)}$,
and $a_{i0}>0$, which gives estimates only for $\na(u_i^{(\tau)})^{1/2}$.
In the latter case, we need to exploit the uniform bound for 
$\na(u_i^{(\tau)}u_j^{(\tau)})^{1/2}$ to be able to pass
to the limit in the expression $A_{ij}(u^{(\tau)})\na u_j^{(\tau)}$.

{\em Step 3: Limit $\delta\to 0$.}
For the limit $\delta\to 0$, we proceed as in \cite{Fis15}.
The idea is to truncate $u_i$ by a smooth function $\varphi_i^L(u)$ that
equals $u_i$ if $\sum_{j=1}^n u_j<L$ and which is constant if 
$\sum_{j=1}^n u_j>2L$.
A discrete version of the entropy-production inequality
\eqref{1.epi1} or \eqref{1.epi2}
(i.e.\ using the test function of the type 
$\sum_{j=1}^n\pa_j\varphi_i^L(u^{(\delta)})\psi$, where $\pa_j=\pa/\pa u_j$
and $\psi\in C^\infty$) gives $\delta$-uniform estimates. 
The integral involving the reaction term
$$
  \sum_{j=1}^n\pa_j\varphi_i^L(u^{(\delta)})
	\frac{f_i(u^{(\delta)})\psi}{1+\delta|f(u^{(\delta)})|}
$$
can be bounded independently of $\delta$ since the support of 
$\pa_j\varphi_i^L$ is bounded. The Aubin-Lions lemma shows that a subsequence of
$(\varphi_i^L(u^{(\delta)}))$ converges strongly in $L^2$ to some function $v_i^L$; 
see the proof of the key Lemma \ref{lem.limdelta}. 
By a diagonal argument, which will be made explicit in the proof of 
Lemma \ref{lem.limdelta}, the subsequence is independent of $L$. The properties of
$\varphi_i^L$ allow us to prove that $u_i^{(\delta)}\to u_i$ a.e.\ in 
$\Omega\times(0,T)$.

The limit $\delta\to 0$ in the equations satisfied by $u^{(\delta)}$ 
yields a defect measure due to the integral involving quadratic gradients of
$u_i^{(\delta)}$ (or $(u_i^{(\delta)})^{1/2}$), 
\begin{equation}\label{1.grad}
  \int_0^T\int_\Omega\pa_j\pa_k\varphi_i^L(u^{(\delta)})A_{j\ell}(u^{(\delta)})
	\na u_\ell^{(\delta)}\cdot\na u_k^{(\delta)}\psi dxdt,
\end{equation}
where $\psi$ is some test function.
As $\delta\to 0$, this expression converges (up to a subsequence) to
$\int_0^T\int_\Omega\psi d\mu_i^L(x,t)$, where $\mu_i^L$ is a signed Radon
measure. 

{\em Step 4: Limit $L\to\infty$.}
It turns out (as in \cite{Fis15}) that $\mu_i^L$ converges weak* to zero as 
$L\to\infty$ in the sense of measures. This follows since the squared gradient in 
\eqref{1.grad} is uniformly bounded, which is a consequence
of the entropy-production inequality \eqref{1.epi1} or \eqref{1.epi2}, and 
$\|\pa_j\pa_k\varphi_i^L\|_{L^\infty}$ converges to zero as $L\to\infty$.
Then we take
$\xi(\varphi_i^L(u))$ as a test function in the equations satisfied by $u_i$,
where the gradient of $\xi\in C^\infty$ has a compact support, and pass to
the limit $L\to\infty$ in the equations. For this step, we use the chain-rule
lemma of \cite{Fis15}.

Compared to \cite{Fis15}, we allow for strongly coupled reaction-diffusion equations
with indefinite diffusion matrices and density-dependent coefficients $A_{ij}(u)$.
Crucial is the linear dependence of $A_{ij}(u)$ on $u_k$.
In fact, when the diffusion coefficients are nonlinear, say of type $u_k^m$ with
$m>0$, we need to define a different entropy density, $h_i(s)=s^m$.
We believe that the existence of renormalized solutions can be proved also 
in this situation, using the ideas of \cite{CDJ17,LeMo17}. 
We expect that the proof can be also extended to cross-diffusion systems
of volume-filling type \cite{Jue16}. Indeed, in this case, the solutions are
bounded such that the expression $A_{ij}(u)\na u_j$ can be easily defined.
Since our proof is already quite technical, we focus on the population model
with linear diffusion coefficients and 
leave details for more general models to the reader.
Finally, we mention that mixed Dirichlet-Neumann boundary conditions may be
treated as well, as long as the entropy-production inequality can be shown;
we refer to \cite{Fis15} for details in the diagonal case.

The paper is organized as follows. The existence of a weak solution to
the approximate problem is shown in Section \ref{sec.approx}. Section \ref{sec.est}
is concerned with the derivation of estimates uniform in $(\eps,\tau)$.
The limit $(\eps,\tau)\to 0$ is proved in Section \ref{sec.epstau},
while the limits $\delta\to 0$ and $L\to\infty$ as well as the proof of Theorem
\ref{thm.ex} are performed in Section \ref{sec.delta}.

%%%%%%%%%%%%%%%%%%%%%%%%%%%%%%%%%%%%%%%%%%%%%%%%%%%%%%%%%%%%%%%%%%%%%%%%%%%%

\section{Existence for an approximate problem}\label{sec.approx}

Let $T>0$, $N\in\N$, set $\tau=T/N$, and let $\delta>0$, and $m\in\N$ with $m>d/2$. 
Then the embedding $H^m(\Omega)\hookrightarrow L^\infty(\Omega)$ is compact.
Let Hypothesis (H3) on the initial datum hold.
To obtain strictly positive initial data, we need to truncate. For this,
let $0<\eps<\min\{1,e^{-\lambda_1},\ldots,
e^{-\lambda_n}\}$ and introduce the cut-off function
$$
  Q_\eps(y) = \left\{\begin{array}{r@{\quad\mbox{if }}l} 
	\eps & 0\le y<\eps,\\
  y & \eps\le y< \eps^{-1},\\
  \eps^{-1} & y\geq \eps^{-1}.
  \end{array}\right.
$$
A computation shows that $h_i(Q_\eps(y))\le e^{-\lambda_i} + h_i(y)$ and
$\lim_{\eps\to 0}Q_\eps(y)=y$ for all $y\ge 0$. We set $u^0_\eps:=
(Q_\eps(u_1^0),\ldots,Q_\eps(u_n^0))$. Then $u^0_\eps(x)\in[\eps,\eps^{-1}]^n$
for $x\in\Omega$, and $w^0=h'(u^0_\eps)\in L^\infty(\Omega;\R^n)$ is
well-defined.

Let $k\ge 1$ and let $w^{k-1}\in L^\infty(\Omega;\R^n)$ be given.
We wish to find $w^k\in H^m(\Omega;\R^n)$ such that for all 
$\phi=(\phi_1,\ldots,\phi_n)\in H^m(\Omega;\R^n)$,
\begin{equation}\label{2.ted}
\begin{aligned}
  \frac{1}{\tau}\int_\Omega & \big(u(w^k)-u(w^{k-1})\big)\cdot\phi dx
	+ \int_\Omega\na\phi:B(w^k)\na w^k dx 
	- \int_\Omega\sum_{i=1}^n u_i(w^k)b_i^k\cdot \na\phi_i dx \\
	&{}+ \eps\int_\Omega\bigg(\sum_{|\alpha|=m}D^\alpha w^k\cdot D^\alpha\phi
	+ w^k\cdot\phi\bigg)dx 
	= \int_\Omega \frac{f(u(w^k))\cdot\phi}{1+\delta |f(u(w^k))|} dx.
\end{aligned}
\end{equation}
Here, $\alpha=(\alpha_1,\ldots,\alpha_n)
\in\N_0^n$ with $|\alpha|=\alpha_1+\cdots+\alpha_n=m$ is a multiindex, and
$D^\alpha=\pa^{|\alpha|}/(\pa x_1^{\alpha_1}\cdots\pa x_n^{\alpha_n})$ is a partial
derivative of order $|\alpha|$. Moreover, 
$u(w^k)=(h')^{-1}(w^k)$ (i.e.\ $u_i(w^k)=\exp(w_i^k-\lambda_i)$),
$B(w^k)=A(u(w^k))h''(u(w^k))^{-1}$, and $b^k=\tau^{-1}\int_{(k-1)\tau}^{k\tau}
b(\cdot,t)dt$. In particular, it follows from (H1) that
$$
  \|b^k\|_{L^\infty(\Omega)} \le \frac{1}{\tau}\int_{(k-1)\tau}^{k\tau}
	\|b(\cdot,t)\|_{L^\infty(\Omega)}dt \le \|b\|_{L^\infty(Q_T)},
$$
where we have set $Q_T=\Omega\times(0,T)$.
First, we show that there exists a solution to \eqref{2.ted}.

\begin{lemma}\label{lem.ted}
Let (H1), (H2), (H4), and either (H5') or (H5'') hold. In case (H5'') with $a_{i0}>0$ for
all $i$, we choose $\tau>0$ sufficiently small.
Then there exists a solution $w^k\in H^m(\Omega;\R^n)$ to \eqref{2.ted}.
\end{lemma}

\begin{proof}
We follow the lines of \cite{CDJ17} but some estimates simplify. For clarity,
we present the full proof.

{\em Step 1.} Let (H5') hold and let $\bar w\in L^\infty(\Omega;\R^n)$. 
We claim that there exists
a unique solution $w=(w_1,\ldots,w_n)\in H^m(\Omega;\R^n)$ to the linear problem
\begin{equation}\label{2.LM}
  a(w,\phi) = F(\phi)\quad\mbox{for all }\phi\in H^m(\Omega;\R^n),
\end{equation}
where
\begin{align*}
  a(w,\phi) &= \tau\int_\Omega\na\phi:B(\bar w)\na w dx
	+ \eps\tau\int_\Omega\bigg(\sum_{|\alpha|=m}D^\alpha w\cdot D^\alpha\phi
	+ w\cdot\phi\bigg)dx, \\
	F(\phi) &= -\int_\Omega\big(u(\bar w)-u(w^{k-1})\big)\cdot\phi dx
	+ \tau\int_\Omega\sum_{i=1}^n u_i(\bar w)b_i^k\cdot\na\phi_i dx \\
	&\phantom{xx}{}+ \tau\int_\Omega\frac{f(u(\bar w))\cdot\phi}{1+\delta|f(u(\bar w))|}dx.
\end{align*}
Since $\bar w$, $w^{k-1}\in L^\infty(\Omega;\R^n)$, we have
$u_i(\bar w)=\exp(\bar w_i-\lambda_i)$, $u_i(w^{k-1})\in L^\infty(\Omega)$.
Then each coefficient of $B(\bar w)$,
$$
  B_{ij}(\bar w) = \delta_{ij}\bigg(a_{i0}+\sum_{k=1}^n a_{ik}u_k(\bar w)\bigg)
	u_i(\bar w) + a_{ij} u_i(\bar w)u_j(\bar w),
$$
is bounded. In view of $b^k_i\in L^\infty(\Omega;\R^n)$, it follows that
the bilinear form $a$ and the linear form $F$ are bounded. The matrix $B(\bar w)$
is positive semidefinite since $h''(u)A(u)$ is positive semidefinite,
by Lemma 6 in \cite{CDJ17}. Consequently, for all $z\in\R^n$,
$$
  z^T B(\bar w)z = \big(h''(u(\bar w))^{-1}z\big)^T h''(u(\bar w))A(u(\bar w))
  \big(h''(u(\bar w))^{-1}z\big) \ge 0.
$$
Hence, we infer from the generalized Poincar\'e inequality 
\cite[Chapter 2, Section 1.4]{Tem97} that the bilinear form $a$ is coercive,
$$
  a(w,w) \ge \tau\eps\int_\Omega\bigg(\sum_{|\alpha|=m}|D^\alpha w|^2
	+ |w|^2\bigg)dx \ge \tau\eps C\|w\|_{H^m(\Omega)}^2
$$
for $w\in H^m(\Omega;\R^n)$. By the Lax-Milgram lemma, we conclude the existence
of a unique solution $w\in H^m(\Omega;\R^n)$ to \eqref{2.LM}. 

{\em Step 2.} Define the mapping $\Phi:L^\infty(\Omega;\R^n)\to L^\infty(\Omega;\R^n)$
by $\Phi(\bar w)=w$, where $w\in H^m(\Omega;\R^n)$ is the unique solution to
\eqref{2.LM}. Standard arguments (see, for instance, the proof of Lemma 5
in \cite{Jue15}), together with Hypothesis (H2), 
show that $\Phi$ is continuous. Then the compactness of the
embedding $H^m(\Omega)\hookrightarrow L^\infty(\Omega)$ implies the compactness
of $\Phi$. In order to apply the Leray-Schauder fixed-point theorem
\cite[Theorem 10.3]{GiTr01}, it remains to show that the set
$\Lambda=\{w\in L^\infty(\Omega;\R^n):w=\sigma \Phi(w)$ for some $\sigma\in(0,1]\}$
is bounded in $L^\infty(\Omega;\R^n)$.

Let $w\in\Lambda$. Then $a(w,\phi)=\sigma F(\phi)$ for all $\phi\in H^m(\Omega;\R^n)$,
with $\bar w$ is replaced by $w$.
Taking $\phi=w$ as a test function, we find that 
\begin{align}
  \tau\int_\Omega & \na w:B(w)\na w dx + \eps\tau\int_\Omega\bigg(
	\sum_{|\alpha|=m}|D^\alpha w|^2 + |w|^2\bigg) \nonumber \\
  &= -\sigma\int_\Omega\big(u(w)-u(w^{k-1})\big)\cdot wdx
	+ \sigma\tau\int_\Omega\frac{f(u(w))\cdot w}{1+\delta|f(u(w))|}dx \label{2.aux} \\
	&\phantom{xx}{}+ \sigma\tau\int_\Omega\sum_{i=1}^n u_i(w)b_i^k\cdot\na w_i dx.
	\nonumber
\end{align}
We estimate both sides term by term.

The identities $B(w)=A(u(w))h''(u(w))^{-1}$ and $\na w=h''(u(w))\na u(w)$ imply that
$$
  \tau\int_\Omega\na w:B(w)\na w dx
	= \tau\int_\Omega\na u(w):h''(u(w))A(u(w))\na u(w)dx.
$$
It is shown in \cite[Lemma 6]{CDJ17} that the matrix $h''(u(w))A(u(w))$ is 
positive definite:
$$
  J_1 := \tau\int_\Omega\na w:B(w)\na w dx 
	\ge 2\eta\tau\int_\Omega\sum_{i=1}^n|\na u_i(w)|^2 dx
$$
for some constant $\eta>0$.
The convexity of $h$ implies that $h(y)-h(z)\le h'(y)\cdot(y-z)$ for all
$x$, $y\in(0,\infty)^n$. Choosing $y=u(w)$, $z=u(w^{k-1})$ and using the
property $h'(u(w))=w$, we infer that
$$
  -\sigma\int_\Omega\big(u(w)-u(w^{k-1})\big)\cdot wdx
	\le -\sigma\int_\Omega\big(h(u(w))-h(u(w^{k-1}))\big)dx.
$$
By Hypothesis (H4), we have
$$
  \sigma\tau\int_\Omega\frac{f(u(w))\cdot w}{1+\delta|f(u(w))|}dx
	= \sigma\tau\int_\Omega\sum_{i=1}^n\frac{f_i(u(w))(\log u_i(w)+\lambda_i)}{
	1+\delta|f(u(w))|}dx \le 0.
$$

It remains to estimate the last integral in \eqref{2.aux}.
We use the Cauchy-Schwarz inequality to find that
\begin{align*}
  J_2 &:= \sigma\tau\int_\Omega \sum_{i=1}^n u_i(w)b_i^k\cdot\na w_i dx
	=  \sigma\tau\int_\Omega \sum_{i=1}^n b_i^k\cdot\na u_i(w) dx \\
	&\le \eta\tau\int_\Omega\sum_{i=1}^n|\na u_i(w)|^2 dx
	+ \frac{\tau}{4\eta}\|b\|_{L^\infty(Q_T)}^2.
\end{align*}
Hence, we conclude from \eqref{2.aux} that
\begin{equation}\label{2.epi1}
\begin{aligned}
  \sigma\int_\Omega h(u(w))dx
	&+ \eta\tau\int_\Omega\sum_{i=1}^n|\na u_i(w)|^2 dx
	+ \eps\tau\int_\Omega\bigg(\sum_{|\alpha|=m}|D^\alpha w|^2 + |w|^2\bigg) \\
	&\le \sigma\int_\Omega h(u(w^{k-1}))dx + C(\eta,b)\tau.
\end{aligned}
\end{equation}
Recalling that $u(w^{k-1})\in L^\infty(\Omega;\R^n)$, we infer
that $\int_\Omega h(u(w^{k-1}))dx<\infty$ and hence
$\|w\|_{L^\infty(\Omega)}$ $\le C\|w\|_{H^m(\Omega)}\le C(\eta,\eps,b,\tau)$,
which is the desired uniform bound. By the Leray-Schauder theorem,
there exists a solution $w^k:=w\in H^m(\Omega;\R^n)$ to \eqref{2.ted}.

{\em Step 3.}
Finally, let (H5'') hold. The proof is almost identical to the previous
steps except the estimation of $J_1$ and $J_2$. 
The matrix $h''(u)A(u)$ is positive definite
and for all $z\in\R^n$ (see \cite[Lemma 4]{CDJ17}),
$$
  z:h''(u)A(u)z \ge \sum_{i=1}^n\pi_i a_{i0}\frac{z_i^2}{u_i}
	+ 2\sum_{i=1}^n\pi_i a_{ii} z_i^2
	+ \frac12\sum_{i,j=1,\ i\neq j}^n\pi_i a_{ij}
	\bigg(\sqrt{\frac{u_j}{u_i}}z_i + \sqrt{\frac{u_i}{u_j}}z_j	\bigg)^2.
$$
Therefore, in case $a_{ii}>0$ for all $i=1,\ldots,n$, we can proceed as in the previous
step, obtaining inequality \eqref{2.epi1}. In case $a_{i0}>0$ for all $i=1,\ldots,n$,
there exists $\eta>0$, only depending on the coefficients $a_{ij}$, such that
$$
  J_1 \ge 4\eta\tau\int_\Omega\sum_{i=1}^n|\na u_i(w)^{1/2}|^2 dx
	+ 2\eta\tau\int_\Omega\sum_{i,j=1,\ i\neq j}^n
	\big|\na(u_i(w)u_j(w))^{1/2}\big|^2 dx.
$$

We use the Cauchy-Schwarz inequality and the elementary inequality
$u_i\le h_i(u_i) + C$ for some constant $C>0$ to infer that
\begin{align*}
  J_2 &= 2\sigma\tau\int_\Omega\sum_{i=1}^n u_i(w)^{1/2} b_i^k\cdot\na u_i(w)^{1/2} dx \\
	&\le \eta\tau\int_\Omega\sum_{i=1}^n|\na u_i(w)^{1/2}|^2 dx
	+ \frac{\sigma\tau}{\eta}\|b\|_{L^\infty(Q_T)}^2
	\int_\Omega\sum_{i=1}^n u_i(w)dx \\
	&\le \eta\tau\int_\Omega\sum_{i=1}^n|\na u_i(w)^{1/2}|^2 dx
	+  C(\eta,b)\sigma\tau\int_\Omega h(u(w))dx + C(\eta,b)\tau.
\end{align*}
Then \eqref{2.aux} can be written as
\begin{align}
  &\sigma(1-C(\eta,b)\tau) \int_\Omega h(u(w))dx 
	+ \eta\tau\int_\Omega\sum_{i=1}^n|\na u_i(w)^{1/2}|^2 dx \nonumber \\
	&\phantom{xx}{}+ 2\eta\tau\sum_{i,j=1\,i\neq j}^n\big|\na(u_i(w)u_j(w))^{1/2}
	\big\|_{L^2(\Omega)}^2 
	+ \eps\tau\bigg(\sum_{|\alpha|=m}\|D^\alpha w\|^2_{L^2(\Omega)} 
	+ \|w\|^2_{L^2(\Omega)}\bigg) \label{2.epi2} \\
	&\le \sigma\int_\Omega h(u(w^{k-1}))dx + C(\eta,b)\tau, \nonumber
\end{align}
and we obtain the desired $L^\infty$ bound for $w$ by choosing $\tau<1/C(\eta,b)$.
This ends the proof.
\end{proof}

%%%%%%%%%%%%%%%%%%%%%%%%%%%%%%%%%%%%%%%%%%%%%%%%%%%%%%%%%%%%%%%%%%%%%%%%%%%%%%%

\section{Uniform estimates}\label{sec.est}

The next step is the derivation of estimates which are uniform in the
approximation parameters. Let (H1)-(H4) and either (H5') or (H5'') hold.
Applying Lemma \ref{lem.ted} iteratively, we obtain
a sequence of solutions $w^k\in H^m(\Omega;\R^n)$ to \eqref{2.ted} with
$u^k:=u(w^k)\in L^\infty(\Omega;(0,\infty)^n)$ for $k=1,\ldots,N$. The first bounds
are a consequence of the discrete entropy estimate \eqref{2.epi1} or \eqref{2.epi2},
respectively.

\begin{lemma}\label{lem.epi}
{\rm (i)} Let (H5') or (H5'') with $a_{ii}>0$ for all $i=1,\ldots,n$ hold 
and let $k\in\{1,\ldots,N\}$. Then
\begin{align*}
  \int_\Omega h(u^k)dx 
	&+ \eta\tau\sum_{j=1}^k\sum_{i=1}^n\|\na u_i^j\|_{L^2(\Omega)}^2 \\
	&{}+ \eps\tau\sum_{j=1}^k\bigg(\sum_{|\alpha|=m}\|D^\alpha w^j\|_{L^2(\Omega)}^2
	+ \|w^j\|_{L^2(\Omega)}^2\bigg)
	\le C(u^0,b,T).
\end{align*} 
{\rm (ii)} Let $k\in\{1,\ldots,N\}$.
If (H5'') with $a_{i0}>0$ for all $i=1,\ldots,n$ holds, then for sufficiently small
$\tau>0$,
\begin{align*}
  \int_\Omega h(u^k)dx 
	&+ \eta\tau\sum_{j=1}^k\sum_{i=1}^n\|\na (u_i^j)^{1/2}\|_{L^2(\Omega)}^2 
	+ \eta\tau\sum_{j=1}^k\sum_{i,\ell=1,\,i\neq\ell}^n
	\|\na(u_i^ju_\ell^j)^{1/2}\|_{L^2(\Omega)}^2\\
	&{}+ \eps\tau\sum_{j=1}^k\bigg(\sum_{|\alpha|=m}\|D^\alpha w^j\|_{L^2(\Omega)}^2
	+ \|w^j\|_{L^2(\Omega)}^2\bigg)
	\le C(u^0,b,T).
\end{align*} 
In both cases, $\eta>0$ and $C(u^0,b,T)>0$ are constants which
are independent of $\delta$, $\eps$, and $\tau$.
\end{lemma}

\begin{proof}
(i) We have $a_{ii}>0$ for all $i=1,\ldots,n$. 
We take $\sigma=1$, $w=w^k$ in \eqref{2.epi1} and sum the equations. This yields
\begin{align*}
  \int_\Omega h(u^k)dx 
	&+ \eta\tau\sum_{j=1}^k\sum_{i=1}^n\|\na u_i^j\|_{L^2(\Omega)}^2 \\
	&{}+ \eps\tau\sum_{j=1}^k\bigg(\sum_{|\alpha|=m}\|D^\alpha w^j\|_{L^2(\Omega)}^2
	+ \|w^j\|_{L^2(\Omega)}^2\bigg) \le \int_\Omega h(u(w^0))dx + C(\eta,b)T.
\end{align*} 
Since $h_i(Q_\eps(y))\le e^{-\lambda_i}+h_i(y)$ on $[0,\infty)$, we have
$$
  h(u(w^0)) = h(u_\eps^0) = \sum_{i=1}^n h_i(Q_\eps(u_i^0))
	\le h(u^0) + C,
$$
concluding the proof.

(ii) Let $a_{i0}>0$ for $i=1,\ldots,n$. We choose $\sigma=1$, $\tau<1/(2C(\eta,b))$, 
and $w=w^k$ in \eqref{2.epi2}
and sum the equations, yielding
\begin{align*}
  (1&-C(\eta,b)\tau)\int_\Omega h(u^k)dx 
	+ \eta\tau\sum_{j=1}^k\sum_{i=1}^n\|\na (u_i^j)^{1/2}\|_{L^2(\Omega)}^2 \\
	&\phantom{xx}{}+ 2\eta\tau\sum_{i,\ell=1,\,i\neq \ell}^n
	\big\|\na(u_i^j u_\ell^j)^{1/2}\big\|_{L^2(\Omega)}^2 
	+ \eps\tau\sum_{j=1}^k\bigg(\sum_{|\alpha|=m}
	\|D^\alpha w^j\|_{L^2(\Omega)}^2 + \|w^j\|_{L^2(\Omega)}^2\bigg) \\
	&\le (1-C(\eta,b)\tau)\int_\Omega h(u(w^0))dx 
	+ C(\eta,b)\tau\sum_{j=1}^k\int_\Omega h(u^{j-1})dx	+ C(\eta,b)T \\
	&\le \int_\Omega h(u^0) + C(\eta,b)\tau\sum_{j=1}^k\int_\Omega h(u^{j-1})dx
	+ C(\eta,b)T.
\end{align*}
Observing that $1-C(\eta,b)\tau\ge 1/2$ and applying the discrete
Gronwall inequality, this proves (ii).
\end{proof}

We also need a uniform estimate for the discrete time derivative.

\begin{lemma}\label{lem.time}
Let (H1)-(H3) and (H5') or (H5'') hold. Then
\begin{equation}\label{2.time}
  \tau\sum_{k=1}^N\big\|\tau^{-1}(u^k-u^{k-1})\big\|_{H^{m+1}(\Omega)'}^r
	\le C(\delta,u^0,b,T),
\end{equation}
where $r=(d+2)/(d+1)$ if (H5') or (H5'') with $a_{ii}>0$ (case (i)) and
$r=(2d+2)/(2d+1)$ if (H5'') with $a_{i0}>0$ (case (ii)).
\end{lemma}

\begin{proof}
We reformulate \eqref{2.ted} as
\begin{equation}\label{2.weak}
\begin{aligned}
  \int_\Omega&\tau^{-1}(u^k-u^{k-1})\cdot\phi dx 
	+ \int_\Omega\na\phi:A(u^k)\na u^k dx
	- \int_\Omega\sum_{i=1}^n u_i^k b_i^k\cdot\na\phi_i dx \\
	&{}+ \eps\int_\Omega\bigg(\sum_{|\alpha|=m}D^\alpha w^k\cdot D^\alpha\phi
	+ w^k\cdot\phi\bigg)dx
	= \int_\Omega\frac{f(u^k)\cdot\phi}{1+\delta|f(u^k)|}dx,
\end{aligned}
\end{equation}
where $\phi\in H^{m+1}(\Omega;\R^n)\hookrightarrow W^{1,\infty}(\Omega;\R^n)$.
Therefore,
\begin{align}
  \bigg|&\int_\Omega\tau^{-1}(u^k-u^{k-1})\cdot\phi dx\bigg|
	\le \sum_{i,j=1}^n\|A_{ij}(u^k)\na u^k_j\|_{L^1(\Omega)}
	\|\na\phi_i\|_{L^\infty(\Omega)} \nonumber \\
	&{}\phantom{xx}+ \|b\|_{L^\infty(Q_T)}\|u^k\|_{L^1(\Omega)}
	\|\na \phi\|_{L^\infty(\Omega)}
	+ \eps\|w^k\|_{H^m(\Omega)}\|\phi\|_{H^m(\Omega)}
	+ \delta^{-1}\|\phi\|_{L^1(\Omega)} \label{2.aux2} \\
	&\le \bigg(\sum_{i,j=1}^n \|A_{ij}(u^k)\na u^k_j\|_{L^1(\Omega)}
	+ C(b)\|u^k\|_{L^1(\Omega)} + \eps\|w^k\|_{H^m(\Omega)} + C(\delta)\bigg)
	\|\phi\|_{H^{m+1}(\Omega)}. \nonumber
\end{align}
Observe that the entropy controls the $L^1$ norm such that, by
Lemma \ref{lem.epi}, 
\begin{equation}\label{2.L1}
  \|u^k\|_{L^1(\Omega)}\le C \quad\mbox{for all }k=1,\ldots,N
\end{equation}
and some $C>0$ which is independent of $k$, $\delta$, $\eps$, and $\tau$.

We estimate now the $L^1$ norm of $A_{ij}(u^k)\na u_j^k$.  
For this, we need to distinguish the cases (i) and (ii). 
In case (i), we take $\theta=d(d+2)\in(0,1)$. 
Then the Gagliardo-Nirenberg inequality and estimate \eqref{2.L1} give
$$
  \|u^k\|_{L^2(\Omega)} \le C\|\na u^k\|_{L^2(\Omega)}^\theta
	\|u^k\|_{L^1(\Omega)}^{1-\theta} + \|u^k\|_{L^1(\Omega)}
	\le C\big(1 + \|\na u^k\|_{L^2(\Omega)}^\theta\big).
$$
Since $A_{ij}(u^k)$ depends linearly on $u_i^k$, this shows that
$$
  \sum_{i,j=1}^n\|A_{ij}(u^k)\na u_j^k\|_{L^1(\Omega)}
	\le \sum_{i,j=1}^n\|A_{ij}(u^k)\|_{L^2(\Omega)}\|\na u_j^k\|_{L^2(\Omega)}
	\le C\big(1 + \|\na u^k\|_{L^2(\Omega)}^{1+\theta}\big).
$$
We deduce from \eqref{2.aux2} that
$$
  \big\|\tau^{-1}(u^k-u^{k-1})\|_{H^{m+1}(\Omega)'}
  \le C\|\na u^k\|_{L^2(\Omega)}^{1+\theta} + \eps\|w^k\|_{H^m(\Omega)} 
	+ C(\delta,b).
$$
For $r=(d+2)/(d+1)$, we have $(1+\theta)r=2$ and $r<2$ and consequently, after summing
\eqref{2.aux2} from $k=1,\ldots,N$,
\begin{align*}
  \bigg(\tau&\sum_{k=1}^N\big\|\tau^{-1}(u^k-u^{k-1})\big\|_{H^{m+1}(\Omega)'}^r
	\bigg)^{1/r} \\
	&\le C\bigg(\tau\sum_{k=1}^N\|\na u^k\|_{L^2(\Omega)}^{(1+\theta)r}\bigg)^{1/r}
	+ \eps\bigg(\tau\sum_{k=1}^N\|w^k\|_{H^m(\Omega)}^r\bigg)^{1/r}
	+ C(\delta,b,T) \\
  &\le C\bigg(\tau\sum_{k=1}^N\|\na u^k\|_{L^2(\Omega)}^{2}\bigg)^{1/r}
	+ C(T)\eps\bigg(\tau\sum_{k=1}^N\|w^k\|_{H^m(\Omega)}^2\bigg)^{1/2}
	+ C(\delta,b,T).
\end{align*}
Lemma \ref{lem.epi} (i) implies that the right-hand side is uniformly bounded,
which shows \eqref{2.time}.

In case (ii), we need the $L^2$ bound for $\na (u_i^k u_j^k)^{1/2}$ and
the special structure of $A_{ij}(u^k)$. Since a similar argument was presented
in \cite[Remark 12]{CDJ17}, we give only a sketch of the proof. First, we observe that  
$$
  \sum_{j=1}^n A_{ij}(u^k)\na u_j^k
	= \na\bigg(a_{i0}u_i^k + \sum_{j=1}^n a_{ij}u_i^ku_j^k\bigg).
$$
Lemma \ref{lem.epi} (ii) shows that 
$\tau\sum_{k=1}^N\|\na(u_i^ku_j^k)^{1/2}\|_{L^2(\Omega)}^2$ and
$\sup_{k=1,\ldots,N}\|(u_i^ku_j^k)^{1/2}\|_{L^1(\Omega)}$ are uniformly bounded.
Hence, by the Gagliardo-Nirenberg inequality,
$\tau\sum_{k=1}^N\|(u_i^ku_j^k)^{1/2}\|_{L^{p}(\Omega)}^p$ is uniformly bounded,
where $p=2+2/d$. Then the
H\"older inequality with $r=(2d+2)/(2d+1)$ and $r'=2d+2$ gives the bound
\begin{align*}
  \bigg(\tau&\sum_{k=1}^N\big\|\na(u_i^k u_j^k)\big\|_{L^r(\Omega)}^r\bigg)^{1/r} 
	= 2\bigg(\tau\sum_{k=1}^n\big\|(u_i^ku_j^k)^{1/2}\na(u_i^ku_j^k)^{1/2}
	\big\|_{L^r(\Omega)}^r\bigg)^{1/r} \\
	&\le 2\bigg(\tau\sum_{k=1}^N\big\|(u_i^ku_j^k)^{1/2}\big\|_{L^{p}(\Omega)}^p
	\bigg)^{1/p}\bigg(\tau\sum_{k=1}^N\big\|\na(u_i^ku_j^k)^{1/2}\big\|_{L^2(\Omega)}^2
	\bigg)^{1/2} \le C.
\end{align*}
Consequently,
$$
  \bigg(\tau\sum_{k=1}^N\sum_{i,j=1}^n\|A_{ij}(u^k)\na u_j^k\|_{L^1(\Omega)}^r
	\bigg)^{1/r} \le C(u^0,b,T).
$$
This proves \eqref{2.time} with $r=(2d+2)/(2d+1)$.
\end{proof}

We define the piecewise constant functions in time
$u^{(\tau)}(x,t)=u(w^k(x))$, $w^{(\tau)}(x,t)=w^k(x)$, and $b^{(\tau)}(x,t)=b^k(x)$
for $x\in\Omega$ and $t\in((k-1)\tau,\tau k]$, $k=1,\ldots,N$. Furthermore, we
need the discrete time derivative $\pa_t^{(\tau)} u^{(\tau)}(x,t)
=\tau^{-1}(u(w^k(x))-u(w^{k-1}(x)))$ for $x\in\Omega$, $t\in((k-1)\tau,\tau k]$.
Recall that for $k=1$, we have $u(w^0)=u_\eps^0$.
We conclude from Lemmas \ref{lem.epi} and \ref{lem.time} the following bounds.

\begin{corollary}\label{coro}
{\rm (i)} Let (H5') or (H5'') with $a_{ii}>0$ for $i=1,\ldots,n$ holds. Then
for all $i=1,\ldots,n$:
\begin{align}
  & \|u_i^{(\tau)}\|_{L^\infty(0,T;L^1(\Omega))}
	+ \|u_i^{(\tau)}\|_{L^2(0,T;H^1(\Omega))}
	+ \eps^{1/2}\|w_i^{(\tau)}\|_{L^2(0,T;H^m(\Omega))}
	\le C(u^0,b,T), \label{2i.est} \\
	&\|\pa_t^{(\tau)}u_i^{(\tau)}\|_{L^r(0,T;H^{m+1}(\Omega)')}
	\le C(\delta,u^0,b,T), \quad r=\frac{d+2}{d+1}. \label{2i.time}
\end{align}
{\rm (ii)} Let (H5'') with $a_{i0}>0$ for $i=1,\ldots,n$ holds. Then for all
$i\neq j$,
\begin{align}
  &\|u_i^{(\tau)}\|_{L^\infty(0,T;L^1(\Omega))}
	+ \|(u_i^{(\tau)})^{1/2}\|_{L^2(0,T;H^1(\Omega))} \\
	&\phantom{xx}{}+ \|\na(u_i^{(\tau)}u_j^{(\tau)})^{1/2}\|_{L^2(Q_T)}^2
	+ \eps^{1/2}\|w_i^{(\tau)}\|_{L^2(0,T;H^m(\Omega))}
	\le C(u^0,b,T), \label{2ii.est} \\
	&\|\pa_t^{(\tau)}u_i^{(\tau)}\|_{L^r(0,T;H^{m+1}(\Omega)')}
	\le C(\delta,u^0,b,T), \quad r=\frac{2d+2}{2d+1}. \label{2ii.time}
\end{align}
\end{corollary}

%%%%%%%%%%%%%%%%%%%%%%%%%%%%%%%%%%%%%%%%%%%%%%%%%%%%%%%%%%%%%%%%%%%%%%%%%%%%%%

\section{The limit $(\eps,\tau)\to 0$}\label{sec.epstau}

The uniform bounds of Corollary \ref{coro} are sufficient to pass to the
simultaneous limit $(\eps,\tau)\to 0$. First, consider case (i), i.e.\
(H5') or (H5'') with $a_{ii}>0$. The gradient bound in
\eqref{2i.est} and estimate \eqref{2i.time} of the discrete time derivative
allow us to apply the Aubin-Lions lemma in the version of \cite{DrJu12},
yielding the existence of a subsequence, which is not relabeled, such that,
as $(\eps,\tau)\to 0$,
\begin{align*}
  u^{(\tau)} \to u &\quad\mbox{strongly in }L^2(Q_T), \\
  u^{(\tau)} \rightharpoonup u &\quad\mbox{weakly in }L^2(0,T;H^1(\Omega)),
\end{align*}
where $u=(u_1,\ldots,u_n)$ and $u_i\ge 0$.
In Case (ii), i.e.\ (H5'') with $a_{i0}>0$,
we have only a gradient estimate for the square root of $u_i^{(\tau)}$,
by \eqref{2ii.est}. Therefore,
we need to apply the nonlinear Aubin-Lions lemma of \cite[Theorem 3]{CJL14}
to find that, again for a subsequence,
\begin{align*}
  u^{(\tau)} \to u &\quad\mbox{strongly in }L^1(Q_T), \\
	(u^{(\tau)})^{1/2}\rightharpoonup u^{1/2} 
	&\quad\mbox{weakly in }L^2(0,T;H^1(\Omega)).
\end{align*}
In both cases, possibly for a subsequence,
$$
  \eps w^{(\tau)}\to 0 \quad\mbox{strongly in }L^2(0,T;H^m(\Omega)).
$$

{\em In the following, we focus on the case (H5') or (H5'') with $a_{ii}>0$ since
the other case, (H5'') with $a_{i0}>0$, can be presented in a similar way.} 
In fact, the existence proof works as long as $A_{ij}(u)\na u_j$
is bounded in $L^s(Q_T)$ for some $s>1$, and this holds in both cases, as shown
in the proof of Lemma \ref{lem.time}.

\begin{lemma}\label{lem.delta}
Let (H5') or (H5'') with $a_{ii}>0$ for $i=1,\ldots,n$ hold.
Then
$$
  \int_\Omega h(u)dx + \|u\|_{L^2(0,T;H^1(\Omega))}
	+ \|u\|_{L^{2+2/d}(Q_T)} + \|A(u)\na u\|_{L^{s}(Q_T)}
	\le C(u^0,b,T),
$$
where $s=(2d+2)/(2d+1)$, $\pa_t u_i\in L^s(0,T;W^{1,2d+2}(\Omega)')$,
for all $\phi\in C_0^\infty(\overline\Omega\times[0,T);\R^n)$,
\begin{equation}\label{2.d}
\begin{aligned}
  -\int_0^T\int_\Omega u\cdot\pa_t\phi dxdt &- \int_\Omega u^0\cdot\phi(\cdot,0)dx 
	+ \int_0^T\int_\Omega\na\phi:(A(u)\na u-ub)dxdt \\
	&=\int_0^T\int_\Omega\frac{f(u)\cdot\phi}{1+\delta|f(u)|}dxdt,
\end{aligned}
\end{equation}
and $u_i(\cdot,0)=u_i^0$ is satisfied in the sense of $W^{1,2d+2}(\Omega)'$.
\end{lemma}

The notation $\psi\in C_0^\infty(\overline\Omega\times[0,T))$ means that
$\psi\in C^\infty(\overline\Omega\times[0,T])$ and $\psi(\cdot,T)=0$.
Furthermore, recall that $\na\phi:(ub)=\sum_{i=1}^n u_ib_i\cdot\na\phi_i$.

\begin{proof}
The strong convergence of $(u^{(\tau)})$ in $L^2(Q_T)$ implies that, for
a subsequence, $u^{(\tau)}\to u$ a.e.\ in $Q_T$ and, because of the continuity
of $h$, $h(u^{(\tau)})\to h(u)$ a.e.\ in $Q_T$. Then Lemma \ref{lem.epi} and
Fatou's lemma show that $\int_\Omega h(u)dx\le C(u^0,b,T)$.
The uniform estimate for $(u^{(\tau)})$ in $L^2(0,T;H^1(\Omega))$ and
the weakly lower semi-continuity of the norm imply that
$\|u\|_{L^2(0,T;H^1(\Omega))}\le C(u^0,b,T)$. Next, we apply the
Gagliardo-Nirenberg inequality with $p=2+2/d$ and $\theta=d/(d+1)$
(such that $\theta p=2$):
\begin{align*}
  \|u_i^{(\tau)}\|_{L^p(Q_T)}^p
	&= \int_0^T\|u_i^{(\tau)}\|_{L^p(\Omega)}^p dx
	\le C\int_0^T\|u_i^{(\tau)}\|_{H^1(\Omega)}^{\theta p}
	\|u_i^{(\tau)}\|_{L^1(\Omega)}^{(1-\theta)p}dt \\
	&\le C\|u_i^{(\tau)}\|_{L^\infty(0,T;L^1(\Omega)}^{(1-\theta)p}
	\int_0^T\|u_i^{(\tau)}\|_{H^1(\Omega)}^2 dt \le C(u^0,b,T).
\end{align*}
Then, by H\"older's inequality, since $1/s=1/p+1/2$,
$$
  \|u_i^{(\tau)}\na u_j^{(\tau)}\|_{L^s(Q_T)} \le C\|u_i^{(\tau)}\|_{L^p(Q_T)}
	\|\na u_i^{(\tau)}\|_{L^2(Q_T)} \le C(u^0,b,T).
$$
Consequently, since $A_{ij}(u^{(\tau)})$ depends on $u_i^{(\tau)}$ linearly,
$\|A_{ij}(u^{(\tau)})\na u_j^{(\tau)}\|_{L^s(Q_T)}\le C(u^0,b,T)$. It holds that
$A_{ij}(u^{(\tau)})\to A_{ij}(u)$ strongly in $L^2(Q_T)$ and
$\na u_j^{(\tau)}\rightharpoonup \na u_j$ weakly in $L^2(Q_T)$. Then the product
converges weakly,
$A_{ij}(u^{(\tau)})\na u^{(\tau)}\rightharpoonup A_{ij}(u)\na u$ weakly in
$L^1(Q_T)$, and this convergence holds even in $L^s(Q_T)$. 
Hence, the limit $(\eps,\tau)\to 0$ leads to the estimate
$\|A_{ij}(u)\na u_j\|_{L^s(Q_T)}\le C(u^0,b,T)$.

Summing the weak formulation \eqref{2.weak} from $k=1,\ldots,N$, we may reformulate
it as
\begin{equation}\label{2.aux3}
\begin{aligned}
  \int_0^T&\int_\Omega \pa_t^{(\tau)}u^{(\tau)}\cdot\phi dxdt
	+ \int_0^T\int_\Omega \na\phi:\big(A(u^{(\tau)})\na u^{(\tau)}
	-u^{(\tau)}b^{(\tau)}\big)dxdt \\
	&{}+ \eps\int_0^T\int_\Omega\bigg(\sum_{|\alpha|=m}D^\alpha w^{(\tau)}
	\cdot D^\alpha\phi + w^{(\tau)}\cdot\phi\bigg)dxdt
	= \int_0^T\int_\Omega\frac{f(u^{(\tau)})\cdot\phi}{1+\delta|f(u^{(\tau)})|}dxdt.
\end{aligned}
\end{equation}
Using the arguments of \cite[pp.~2792-2793]{ChLi13}, we can show that
\begin{align*}
  \int_0^T\int_\Omega\pa_t^{(\tau)}u^{(\tau)}\cdot\phi dxdt
	&\to -\int_0^T\int_\Omega u\cdot\pa_t\phi dxdt - \int_\Omega u^0\cdot\phi(\cdot,0)dx, \\
  \pa_t^{(\tau)}u^{(\tau)}
	&\rightharpoonup \pa_t u \quad\mbox{weakly in }L^r(0,T;H^{m+1}(\Omega)').
\end{align*}
Since $A(u^{(\tau)})\na u^{(\tau)}\rightharpoonup A(u)\na u$
weakly in $L^s(Q_T)$ (see the above argumentation) and $b^{(\tau)}\to b$ in $L^2(Q_T)$,
we can pass to the limit $(\eps,\tau)\to 0$ in the second integral of \eqref{2.aux3}.
The third integral vanishes in the limit, and the integral on the right-hand
side of \eqref{2.aux3} converges to
$$
  \int_0^T\int_\Omega \frac{f(u)\cdot\phi}{1+\delta|f(u)|}dxdt,
$$
since $f(u^{(\tau)})/(1+\delta|f(u^{(\tau)})|)$ is bounded independently of
$(\eps,\tau)$ (but depending on $\delta$). Thus, in the limit $(\eps,\tau)\to 0$,
we infer formulation \eqref{2.d}. Since $A_{ij}(u)\na u_j\in L^s(Q_T)$ and
$f_i(u)/(1+\delta|f(u)|)\in L^\infty(Q_T)$ for any fixed $\delta>0$, a density
argument shows that the weak formulation \eqref{2.d} holds for all
$\phi\in L^{2d+2}(0,T;W^{1,2d+2}(\Omega;\R^n))$. This implies that
$\pa_t u_i\in L^s(0,T;W^{1,2d+2}(\Omega)')$.
\end{proof}

%%%%%%%%%%%%%%%%%%%%%%%%%%%%%%%%%%%%%%%%%%%%%%%%%%%%%%%%%%%%%%%%%%%%%%%%%%%%

\section{The limit $\delta\to 0$, $L\to\infty$}\label{sec.delta}

We use two results from \cite{Fis15}, a truncation and an approximate chain rule. 
For the truncation, let $\varphi\in C^\infty(\R)$ be a
nonincreasing function satisfying $\varphi(x)=1$ for $x<0$ and $\varphi(x)=0$
for $x\ge 1$. We define for $i=1,\ldots,n$ and $L\in\N$ the truncation function
\begin{equation}\label{5.varphi}
  \varphi_i^L(v) = v_i\varphi\bigg(\frac{1}{L}\sum_{k=1}^n v_k - 1\bigg)
	+ 2L\bigg(1 - \varphi\bigg(\frac{1}{L}\sum_{k=1}^n v_k - 1\bigg)\bigg)
	\quad\mbox{for } v\in[0,\infty)^n.
\end{equation}
Clearly, $\varphi_i^L\in C^\infty([0,\infty)^n)$. Moreover, the following
properties hold.

\begin{lemma}
It holds that
\renewcommand{\labelenumi}{\rm (L\theenumi)}
\begin{enumerate}[leftmargin=10mm]
\item %L1
For all $L\in\N$, $v\in[0,\infty)^n$, $i=1,\ldots,n$,
$$
  0\le \varphi_i^L(v)\le v_i + 2\sum_{k=1}^n v_k.
$$

\item %L2
For all $v\in[0,\infty)^n$ with $\sum_{k=1}^n v_k<L$, we have $\varphi_i^L(v)=v_i$.

\item %L3
For any fixed $L\in\N$, $\operatorname{supp}(\varphi_i^L)'$ is a compact subset 
of $[0,\infty)^n$.

\item %L4
For all $v\in[0,\infty)^n$, $j=1,\ldots,n$, we have 
$\lim_{L\to\infty}\pa_j\varphi_i^L(v)=\delta_{ij}$.

\item %L5
There exists $K_1>0$ such that for all $L\in\N$, $v\in[0,\infty)^n$, $j=1,\ldots,n$,
$$
  |\pa_j\varphi_i^L(v)|\le K_1.
$$

\item %L6
For all $j,k=1\ldots,n$,
$$
  \lim_{L\to\infty}\sup_{v\in[0,\infty)^n}|\pa_j\pa_k\varphi_i^L(v)|=0.
$$

\item %L7
There exists $K_2>0$ such that for all $L\in\N$, $v\in[0,\infty)^n$, 
$j,k,\ell=1,\ldots,n$,
$$
  (1+v_\ell)|\pa_j\pa_k\varphi_i^L(v)| + v_j^{1/2}v_k^{1/2}|\pa_j\pa_k\varphi_i^L(v)|
	\le K_2.
$$

\item %L8
Let $L>L_0>0$ and $v\in[0,\infty)^n$. If $\sum_{i=1}^n v_i\ge L_0$ then
$\sum_{i=1}^n\varphi_i^L(v)\ge L_0$.
\end{enumerate}
\end{lemma}

Property (L4) is not used in the proof but it clarifies the role of
$\varphi_i^L$, when we compute the limit $L\to\infty$ in
$\sum_{j=1}^n\pa_j\varphi_i^L(u)$, which gives $u_i$.

\begin{proof}
Ad (L1): Let $v_s:=\frac{1}{L}\sum_{k=1}^n v_k-1$. 
If $v_s<0$ then $\varphi(v_s)=1$ and $\varphi_i^L(v)=v_i$.
If $v_s>1$ (which is equivalent to $2L<\sum_{k=1}^n v_k$) then 
$\varphi(v_s)=0$ and $\varphi_i^L(v) = 2L < \sum_{k=1}^n v_k$. 
Finally, if $0\le v_s\le 1$ (or $L\le \sum_{k=1}^n v_k\le 2L$), we have
$\varphi_i^L(v)\le v_i + 2L \le v_i + 2\sum_{k=1}^n v_k$.

Ad (L2): $\sum_{k=1}^n v_k<L$ implies that $v_s<0$ and $\varphi_i^L(v)=v_i$.

Ad (L3): This is clear since $\varphi$ is constant on $(-\infty,0)$ and $(1,\infty)$.

Ad (L4) and (L5): See (E4) and (E5), respectively, in \cite{Fis15}.

Ad (L6): This is (E7) in \cite{Fis15} except that the supremum is computed
on $(0,\infty)^n$ (which is possible since $\pa_j\pa_k\varphi_i^L$ has compact
support).

Ad (L7): The second inequality can be found in (E2) of \cite{Fis15}.
The first one is new and can be readily verified.

Ad (L8): By definition of $\varphi_i^L$, we have
$$
  \sum_{i=1}^n\varphi_i^L(v) - \sum_{i=1}^n v_i
	= \bigg(2nL - \sum_{i=1}^n v_i\bigg)
	\bigg(1 - \varphi\bigg(\frac{1}{L}\sum_{k=1}^nv_k - 1\bigg)\bigg).
$$
If $\sum_{i=1}^n v_i\le 2nL$ then $\sum_{i=1}^n\varphi_i^L(v)\ge\sum_{i=1}^n v_i\ge L_0$;
otherwise $\frac{1}{L}\sum_{k=1}^n v_k-1>2n-1\ge 1$. Hence,
$\varphi_i^L(v)=2L>L_0$.
\end{proof}

The second result concerns an approximate chain rule. Let 
${\mathcal M}(\overline{\Omega}\times[0,T))$ denote the set of Radon measures on
the Borel sets of $\overline{\Omega}\times[0,T)$ and let $S$ denote the
$(d-1)$-dimensional Hausdorff measure.

\begin{lemma}[Lemma 4 of \cite{Fis15}]\label{lem.chain}
Let $\Omega\subset\R^d$ ($d\ge 1$) be a bounded domain with Lipschitz boundary,
$T>0$, $v_{0,i}\in L^1(\Omega)$, $w_i\in L^1(0,T;L^1(\Omega))$, 
$z_i\in L^2(0,T;L^2(\Omega;\R^n))$,
$q_i\in L^1(0,T;L^1(\pa\Omega))$, 
and $\mu_i\in{\mathcal M}(\overline{\Omega}\times[0,T))$, $i=1,\ldots,n$.
Assume that $v\in L^2(0,T;H^1(\Omega;$ $\R^n))$ solves for all 
$\psi\in C_0^\infty(\overline{\Omega}\times[0,T))$,
\begin{align*}
  -\int_0^T\int_\Omega v_i\pa_t\psi dxdt - \int_\Omega v_{0,i}\psi(\cdot,0)dx 
	&= \int_{\overline{\Omega}\times[0,T)}\psi d\mu_i + \int_0^T\int_\Omega w_i\psi dxdt \\
	&\phantom{xx}{}+ \int_0^T\int_{\pa\Omega}q_i\psi dSdt
	+ \int_0^T\int_\Omega z_i\cdot\na \psi dxdt.
\end{align*}
Finally, let $\xi\in C^\infty(\R^n)$ be given with compactly supported first
derivatives. Then there exists $C(\Omega)>0$ such that 
for all $\psi\in C_0^\infty(\overline{\Omega}\times[0,T))$,
\begin{align*}
  \bigg|&-\int_0^T\int_\Omega\xi(v)\pa_t\psi dxdt - \int_\Omega\xi(v_0)\psi(\cdot,0)dx \\
	&\phantom{xx}{}- \int_0^T\int_\Omega\sum_{i=1}^n\pa_i\xi(v)w_i\psi dxdt
	- \int_0^T\int_{\pa\Omega}\sum_{i=1}^n\pa_i\xi(v)q_i\psi dSdt \\
	&\phantom{xx}{}- \int_0^T\int_\Omega\sum_{i=1}^n\pa_i\xi(v)z_i\cdot\na\psi dxdt
	- \int_0^T\int_\Omega\sum_{j,k=1}^n \pa_j\pa_k\xi(v)z_j\cdot\na v_k\psi dxdt\bigg| \\
	&\le C(\Omega)\sup_u|\xi'(u)|\|\psi\|_{L^\infty(Q_T)}
	\sum_{i=1}^n\|\mu_i\|_{{\mathcal M}(\overline\Omega\times[0,T))}.
\end{align*}
\end{lemma}

We apply this lemma to \eqref{2.d}.

\begin{lemma}\label{lem.varphi}
Let $u^{(\delta)}$ be a weak solution to \eqref{2.d}. Then, for all
$\psi\in C_0^\infty(\overline{\Omega}\times[0,T))$, $L\in\N$, and
$i=1,\ldots,n$,
\begin{align}
  -\int_0^T & \int_\Omega\varphi_i^L(u^{(\delta)})\pa_t\psi dxdt
	- \int_\Omega\varphi_i^L(u^0)\psi(\cdot,0)dx \nonumber \\
	&= -\int_0^T\int_\Omega\sum_{j,k=1}^n\pa_j\pa_k\varphi_i^L(u^{(\delta)})
	\bigg(\sum_{\ell=1}^n A_{j\ell}(u^{(\delta)})\na u_\ell^{(\delta)}
	- u_j^{(\delta)}b_j\bigg)\cdot\na u_k^{(\delta)}\psi dxdt \label{5.weak0} \\
	&\phantom{xx}{}- \int_0^T\int_\Omega\sum_{j=1}^n\pa_j\varphi_i^L(u^{(\delta)})
	\bigg(\sum_{\ell=1}^n A_{j\ell}(u^{(\delta)})\na u_\ell^{(\delta)}
	- u_j^{(\delta)}b_j\bigg)\cdot\na \psi dxdt \nonumber \\
	&\phantom{xx}{}+ \int_0^T\int_\Omega\sum_{j=1}^n\pa_j\varphi_i^L(u^{(\delta)})
	\frac{f_j(u^{(\delta)})\psi}{1+\delta|f(u^{(\delta)})|}dxdt, \nonumber
\end{align}
where we recall definition \eqref{5.varphi} of $\varphi_i^L$.
\end{lemma}

\begin{proof}
Taking the test function
$\phi=(\phi_1,\ldots,\phi_n)$ with $\phi_j=\delta_{ij}\psi$ and
$\psi\in C_0^\infty(\overline{\Omega}\times[0,T))$, we see that $u^{(\delta)}$ solves
\begin{align}
  -&\int_0^T \int_\Omega u_i^{(\delta)}\pa_t\psi dxdt 
	- \int_\Omega u_i^0\psi(\cdot,0)dx \nonumber \\
	&{}+ \int_0^T\int_\Omega\na\psi\cdot\bigg(\sum_{j=1}^n A_{ij}(u^{(\delta)})
	\na u_j^{(\delta)} - u_i^{(\delta)}b_i\bigg)dxdt
	=\int_0^T\int_\Omega\frac{f_i(u^{(\delta)})\psi}{1+\delta|f(u^{(\delta)})|}dxdt.
	\label{5.d}
\end{align}
In Lemma \ref{lem.chain}, we choose $\xi(v)=\varphi_i^L(v)$, $v_i=u_i^{(\delta)}
\in L^2(0,T;H^1(\Omega))$, $v_{0,i}=u_i^0\in L^1(\Omega)$, 
$w_i=f_i(u^{(\delta)})/(1+\delta|f(u^{(\delta)})|)\in L^1(Q_T)$,
$\mu_i=0$, $q_i=0$, and 
$z_i = -(\sum_{j=1}^n A_{ij}(u^{(\delta)})\na u_j^{(\delta)} - u_i^{(\delta)}b_i)
\in L^s(0,T;L^s(\Omega;\R^n))$ with $s=(2d+2)/(2d+1)$, where the regularity 
statements for $u_i^{(\delta)}$ and $z_i$ are obtained from Lemma \ref{lem.delta}.
If $z_i\in L^2(0,T;L^2(\Omega;\R^n))$, the weak formulation
\eqref{5.weak0} is a direct result of Lemma \ref{lem.chain}.
We claim that the lemma can be applied also in the present situation. Indeed,
as the support of $(\varphi_i^L)'$ is bounded and $\na u_\ell^{(\delta)}\in
L^2(Q_T)$, we have
$$
  \pa_j\pa_k\varphi_i^L(u^{(\delta)})z_j
	= \pa_j\pa_k\varphi_i^L(u^{(\delta)})\bigg(\sum_{\ell=1}^n A_{j\ell}(u^{(\delta)})
	\na u_\ell^{(\delta)} - u_j^{(\delta)}b_j\bigg)
	\in L^2(Q_T;\R^n),
$$
and this regularity is sufficient for the proof of Lemma \ref{lem.chain}.
More precisely, let \cite[page 579]{Fis15}
$$
  \widehat v_i = \rho_\eps * u_i^{(\delta)}, \quad
	\widehat w_i = \rho_\eps * w_i + \diver\bigg(\sum_{\ell=1}^n 
	A_{i\ell}(\rho_\eps * u^{(\delta)})\na(\rho_\eps * u_\ell^{(\delta)})
	- (\rho_\eps * u_i^{(\delta)})b_i\bigg),
$$
where $\rho_\eps$ is the standard mollifier on $\Omega$. Then the second line 
on page 580 in \cite{Fis15} can be replaced by
\begin{align*}
  \sum_{j,k=1}^n\int_0^T\int_\Omega & \pa_j\pa_k\varphi_i^L(\rho_\eps * u^{(\delta)})
	\bigg(\sum_{\ell=1}^n A_{j\ell}(\rho_\eps * u^{(\delta)})
	\na(\rho_\eps * u_\ell^{(\delta)}) - (\rho_\eps * u_j^{(\delta)})b_j\bigg) \\
	&{}\times \na(\rho_\eps * u_k^{(\delta)})\psi dxdt.
\end{align*}
The regularity $\na u_\ell^{(\delta)}\in L^2(Q_T)$ and condition (L3) are
sufficient to pass to the limit $\eps\to 0$. The corresponding term on page 582, line 5 
in \cite{Fis15} can be treated in a similar way.
Consequently, Lemma \ref{lem.chain} implies \eqref{5.weak0}.
\end{proof}

\begin{remark}\rm
Note that if (H5'') with $a_{i0}>0$ holds, the argumentation of the previous
proof is slightly different.
Because of the $L^2$ bound of $\na(u_i^{(\delta)})^{1/2}$, we write
$$
   \pa_j\pa_k\varphi_i^L(u^{(\delta)})z_j
	= \pa_j\pa_k\varphi_i^L(u^{(\delta)})\bigg(2\sum_{\ell=1}^n A_{j\ell}(u^{(\delta)})
	(u_\ell^{(\delta)})^{1/2}\na(u_\ell^{(\delta)})^{1/2}
	- u_j^{(\delta)}b_j\bigg),
$$
and this expression is still in $L^2(Q_T;\R^n)$ taking into account 
the compact support of $\pa_k\varphi_i^L$. This argument can be also used in the 
following proofs. 
\qed
\end{remark}

Now, we can perform the limit $\delta\to 0$. The following result is the key lemma.

\begin{lemma}\label{lem.limdelta}
There exists a subsequence of $(u^{(\delta)})$ (not relabeled) and a nonnegative function 
$u\in L^2(0,T;$ $H^1(\Omega;\R^n))$ satisfying $\int_\Omega h(u)dx<\infty$ such that,
as $\delta\to 0$,
\begin{align}
  u^{(\delta)} \to u &\quad\mbox{strongly in }L^2(Q_T), \label{5.strong} \\
	u^{(\delta)} \rightharpoonup u &\quad\mbox{weakly in }L^2(0,T;H^1(\Omega)).
	\label{5.weak}
\end{align}
Moreover, $u$ solves for all $L\in\N$, $i=1,\ldots,n$, and $\psi\in
C_0^\infty(\overline{\Omega}\times[0,T))$,
\begin{align}
  -\int_0^T&\int_\Omega\varphi_i^L(u)\pa_t\psi dxdt 
	- \int_\Omega\varphi_i^L(u^0)\psi(\cdot,0)dx 
	= -\int_0^T\int_\Omega\psi d\mu_i^L(x,t) \nonumber \\
	&{}- \int_0^T\int_\Omega\sum_{j=1}^n\pa_j\varphi_i^L(u)\bigg(\sum_{\ell=1}^n
	A_{j\ell}(u)\na u_\ell - u_jb_j\bigg)\cdot\na\psi dxdt \label{5.weak1} \\
	&{}+ \int_0^T\int_\Omega\sum_{j=1}^n\pa_j\varphi_i^L(u)f_j(u)
	\psi dxdt, \nonumber
\end{align}
where $(\mu_i^L)_{L\in\N}$ is a sequence of signed Radon measures satisfying
\begin{equation}\label{5.mu}
  \lim_{L\to\infty}|\mu_i^L|\big(\overline{\Omega}\times[0,T)\big) = 0.
\end{equation}
\end{lemma}

\begin{proof}
{\em Step 1.} The weak formulation \eqref{5.weak1} holds for all test functions
which vanish at $t=T$. We wish to derive a weak formulation valid for all
test functions $\psi\in C^\infty(\overline\Omega\times[0,T])$. To this end,
we introduce for $m\in\N$ the functions
$$
  g_m(t) = \left\{\begin{array}{l@{\quad\mbox{if }}l} 
	1 & t\in (-1,T-\frac{1}{m}],\\
  2m(T-\frac{1}{m}-t)+1 & t\in (T-\frac{1}{m},T-\frac{1}{2m}),\\
  0 & t\in [T-\frac{1}{2m},T+1).\,
  \end{array}\right.
$$
Then the function $g_m$ is continuous on $(-1,T+1)$, its weak derivative equals
$$
  g'_m(t) = \left\{\begin{array}{l@{\quad\mbox{if }}l}
      -2m & t\in (T-\frac{1}{m},T-\frac{1}{2m}),\\
      0 & t\in (-1,T-\frac{1}{m}]\cup[T-\frac{1}{2m},T+1),
      \end{array}\right.
$$
$\lim_{m\to\infty}g_m(t)=1$ for $t\in[0,T)$, and $\lim_{m\to\infty}g_m(t)=0$
for $t\in[T,T+1)$. Set $g_m^\eps:=\eta_\eps*g_m$,
where $\eta_\eps$ is the standard mollifier on $\R$; see the definition in 
\cite[Section C.4]{Eva10}. In particular (see \cite[Section 5.3.1, Theorem 1]{Eva10}), 
$g_m^\eps\in C_0^\infty([0,T))\cap
C^\infty([0,T])$, $(g_m^\eps)'=\eta_\eps*g_m'$
on $(-1+\eps,T+1-\eps)$ for sufficiently small $\eps>0$, and
\begin{equation}\label{5.eps}
  g_m^\eps\to g_m\quad\mbox{in }C^0([0,T]), \quad
	(g_m^\eps)'\to g_m' \quad\mbox{in }L^2(0,T)\mbox{ as }\eps\to 0.
\end{equation}

Let $\psi\in C^\infty(\overline\Omega\times[0,T])$. Then $\psi g_m^\eps
\in C_0^\infty(\overline\Omega\times[0,T))$, and we can use this function
as a test function in \eqref{5.weak0}:
\begin{align*}
  -\int_0^T & \int_\Omega\varphi_i^L(u^{(\delta)})\pa_t(\psi g_m^\eps) dxdt
	- \int_\Omega\varphi_i^L(u^0)(\psi g_m^\eps)(\cdot,0)dx \\
	&= -\int_0^T\int_\Omega\sum_{j,k=1}^n\pa_j\pa_k\varphi_i^L(u^{(\delta)})
	\bigg(\sum_{\ell=1}^n A_{j\ell}(u^{(\delta)})\na u_\ell^{(\delta)}
	- u_j^{(\delta)}b_j\bigg)\cdot\na u_k^{(\delta)}(\psi g_m^\eps) dxdt \\
	&\phantom{xx}{}- \int_0^T\int_\Omega\sum_{j=1}^n\pa_j\varphi_i^L(u^{(\delta)})
	\bigg(\sum_{\ell=1}^n A_{j\ell}(u^{(\delta)})\na u_\ell^{(\delta)}
	- u_j^{(\delta)}b_j\bigg)\cdot\na \psi g_m^\eps dxdt \\
	&\phantom{xx}{}+ \int_0^T\int_\Omega\sum_{j=1}^n\pa_j\varphi_i^L(u^{(\delta)})
	\frac{f_j(u^{(\delta)})\psi g_m^\eps}{1+\delta|f(u^{(\delta)})|}dxdt.
\end{align*}
Taking into account the compact support of $(\varphi_i^L)'$, by (L3),
the uniform bounds from Lemma \ref{lem.delta}, and the
convergence properties \eqref{5.eps}, we can pass to the limit $\eps\to 0$
in the previous equation, leading to
\begin{align*}
  -\int_0^T & \int_\Omega\varphi_i^L(u^{(\delta)})\pa_t\psi g_m dxdt
	- \int_0^T\int_\Omega\varphi_i^L(u^{(\delta)})\psi g_m' dx
	- \int_\Omega\varphi_i^L(u^0)(\psi g_m)(\cdot,0)dx \\
	&= -\int_0^T\int_\Omega\sum_{j,k=1}^n\pa_j\pa_k\varphi_i^L(u^{(\delta)})
	\bigg(\sum_{\ell=1}^n A_{j\ell}(u^{(\delta)})\na u_\ell^{(\delta)}
	- u_j^{(\delta)}b_j\bigg)\cdot\na u_k^{(\delta)}(\psi g_m) dxdt \\
	&\phantom{xx}{}- \int_0^T\int_\Omega\sum_{j=1}^n\pa_j\varphi_i^L(u^{(\delta)})
	\bigg(\sum_{\ell=1}^n A_{j\ell}(u^{(\delta)})\na u_\ell^{(\delta)}
	- u_j^{(\delta)}b_j\bigg)\cdot\na (\psi g_m) dxdt \\
	&\phantom{xx}{}+ \int_0^T\int_\Omega\sum_{j=1}^n\pa_j\varphi_i^L(u^{(\delta)})
	\frac{f_j(u^{(\delta)})\psi g_m}{1+\delta|f(u^{(\delta)})|}dxdt.
\end{align*}

Next, we perform the limit $m\to\infty$. The only delicate term is the integral
involving $g_m'$:
\begin{align*}
  -\int_0^T\int_\Omega \varphi_i^L(u^{(\delta)})\psi g_m'dxdt
  &= 2m\int_{T-1/m}^{T-1/(2m)}\int_\Omega \varphi_i^L(u^{(\delta)})\psi dxdt \\
  &\to\int_\Omega \varphi_i^L(u^{(\delta)}(\cdot,T))\psi(\cdot,T)dx
	\quad\mbox{as }m\to\infty.
\end{align*} 
For the other terms, we employ the uniform bounds in Lemma \ref{lem.delta}, 
the pointwise
convergence $g_m(t)\to 1$ for $t\in[0,T)$, and Lebesgue's dominated convergence
theorem. Then, in the limit $m\to\infty$,
\begin{align}
  -\int_0^T&\int_\Omega \varphi_i^L(u^{(\delta)})\pa_t \psi dxdt
	+ \int_\Omega \varphi_i^L(u^{(\delta)}(\cdot,T))\psi(\cdot,T)dx
  - \int_\Omega \varphi_i^L(u^0)\psi(,\cdot0) dx \nonumber \\
  &= -\int_0^T\int_\Omega\sum_{j,k=1}^n \pa_j\pa_k\varphi_i^L(u^{(\delta)})
	\bigg(\sum_{\ell=1}^nA_{j\ell}(u^{(\delta)})\na u^{(\delta)}_\ell
	- u^{(\delta)}_j b_j\bigg)\cdot\na u^{(\delta)}_k\psi dxdt \label{5.aux} \\
  &\phantom{xx}{}- \int_0^T\int_\Omega\sum_{j=1}^n \pa_j\varphi_i^L(u^{(\delta)})
	\bigg(\sum_{\ell=1}^nA_{j\ell}(u^{(\delta)})\na u^{(\delta)}_\ell
	- u^{(\delta)}_j b_j\bigg)\cdot\na \psi dxdt \nonumber \\
  &\phantom{xx}{}+ \int_0^T\int_\Omega\sum_{j=1}^n \pa_j\varphi_i^L(u^{(\delta)})
	\frac{f_j(u^{(\delta)})\psi}{1+\delta|f(u^{(\delta)})|}dxdt. \nonumber
\end{align}
This holds for all $\psi\in C^\infty(\overline\Omega\times[0,T])$.
In fact, by a density argument, the weak formulation also holds for all
$\psi\in C^0([0,T];H^p(\Omega))$, where $p>d/2+1$ (such that the embedding
$H^p(\Omega)\hookrightarrow W^{1,\infty}(\Omega)$ is continuous).

{\em Step 2.} We claim that a subsequence of 
$(\varphi_i^L(u^{(\delta)}))$ is convergent
in the limit $\delta\to 0$. Observing that the dual space of 
$C^0([0,T];H^p(\Omega))$ is ${\mathcal M}([0,T];H^p(\Omega)')$, where
${\mathcal M}$ denotes the space of Radon measures, we find that 
\begin{align*}
  \|\pa_t \varphi_i^L(u^{(\delta)})\|_{{\mathcal M}([0,T];H^p(\Omega)')}
	&= \sup_{\|\psi\|_{C^0([0,T];H^p(\Omega))}\le 1}
	\bigg|-\int_0^T\int_\Omega \varphi_i^L(u^{(\delta)})\pa_t \psi dxdt \\
	&\phantom{xx}{}+ \int_\Omega \varphi_i^L(u^{(\delta)}(\cdot,T))\psi(\cdot,T)dx
  - \int_\Omega \varphi_i^L(u^0)\psi(\cdot,0) dx \bigg|.
\end{align*}
We insert \eqref{5.aux} and take into account the uniform bounds
in Lemma \ref{lem.delta} and the compact support of $(\varphi_i^L)'$.
Then the right-hand side of \eqref{5.aux} can be bounded uniformly in $\delta$:
\begin{equation}\label{5.AL1}
  \|\pa_t \varphi_i^L(u^{(\delta)})\|_{{\mathcal M}([0,T];H^p(\Omega)')}
  \le C(L,u^0,b,T).
\end{equation}
By (L1) and (L5), the function $\varphi_i^L$ is growing at most linearly 
and its gradient is bounded. 
Therefore, the gradient bound for $u^{(\delta)}$ shows that
\begin{equation}\label{5.AL2}
  \|\varphi_i^L(u^{(\delta)})\|_{L^2(0,T;H^1(\Omega))} \le C(u^0,b,T).
\end{equation}
Estimates \eqref{5.AL1} and \eqref{5.AL2} allow us to apply the Aubin-Lions
lemma in the version of \cite[Section 7.3, Corollary 7.9]{Rub05} to obtain
the existence of a subsequence of $(\varphi_i^L(u^{(\delta)}))$, which is
not relabeled, such that, as $\delta\to 0$,
\begin{equation}\label{5.conv}
  \varphi_i^L(u^{(\delta)})\to v_i^L\quad\mbox{strongly in }L^2(Q_T)
\end{equation}
for some nonnegative function $v_i^L$.

We claim that the subsequence can be chosen in such a way
that it is independent of $i\in\{1,\ldots,n\}$ and $L\in\N$. Since the set
$\{1,\ldots,n\}$ is finite, we have to prove this statement only for $L\in\N$.
The idea is to apply a diagonal argument. Let $i\in\{1,\ldots,n\}$ be fixed.
By the Aubin-Lions lemma,
there exists a subsequence $(\delta^{1}_k)$ such that
$\varphi_i^1(u^{(\delta^{1}_k)})\to v_i^1$ as $k\to\infty$, and there
exists a subsequence $(\delta_k^{2})$ of $(\delta_k^{1})$ such that
$\varphi_i^2(u^{(\delta_k^{2})})\to v_i^2$. Continuing this argument, 
we find a subsequence $(\delta^{L}_k)$ of $(\delta^{L-1}_k)$ such that
$\varphi_i^L(u^{(\delta^{L}_k)})\to v_i^L$. In other words,
\begin{align*}
  \varphi_i^1(u^{(\delta^{1}_1)}),\
	\varphi_i^1(u^{(\delta^{1}_2)}),\
	\varphi_i^1(u^{(\delta^{1}_3)}),\ldots &\to v_i^1, \\
  \varphi_i^1(u^{(\delta^{2}_1)}),\
	\varphi_i^1(u^{(\delta^{2}_2)}),\
	\varphi_i^1(u^{(\delta^{2}_3)}),\ldots &\to v_i^1, \\	
  \varphi_i^1(u^{(\delta^{3}_1)}),\
	\varphi_i^1(u^{(\delta^{3}_2)}),\
	\varphi_i^1(u^{(\delta^{3}_3)}),\ldots &\to v_i^1,\mbox{ etc.}
\end{align*}
Thus, taking the diagonal terms as a new subsequence
$(\varphi_i^1(u^{(\delta^{k}_k)}))$, we have the convergence
$\varphi_i^1(u^{(\delta^{k}_k)})\to v_i^1$ as $k\to\infty$.
For $L=2$, we argue in a similar way,
\begin{align*}
  \varphi_i^2(u^{(\delta^{2}_2)}),\
	\varphi_i^2(u^{(\delta^{2}_3)}),\
	\varphi_i^2(u^{(\delta^{2}_4)}),\ldots &\to v_i^2, \\
  \varphi_i^2(u^{(\delta^{3}_2)}),\
	\varphi_i^2(u^{(\delta^{3}_3)}),\
	\varphi_i^2(u^{(\delta^{3}_4)}),\ldots &\to v_i^2, \\	
  \varphi_i^2(u^{(\delta^{4}_2)}),\
	\varphi_i^2(u^{(\delta^{4}_3)}),\
	\varphi_i^2(u^{(\delta^{4}_4)}),\ldots &\to v_i^2,\mbox{ etc.}
\end{align*}
Then the diagonal sequence converges, $\varphi_i^2(u^{(\delta^{k}_k)})\to v_i^2$
as $k\to\infty$. This argument can be continued, and we obtain a universal
subsequence $(u^{(\delta^{k}_k)})$, which is independent of $L$, such that
\eqref{5.conv} holds, and we call this subsequence simply $u^{(\delta)}$.

{\em Step 3.} We prove that, up to a subsequence, $u_i^{(\delta)}\to u_i$ a.e.
First, we claim that $(v_i^L)_{L\in\N}$ is a Cauchy sequence.
Let $K$, $L\in\N$ with $K>L$ be given. Then, by definition of $\varphi_i^L$ and
setting $v_\delta^L:=\frac{1}{L}\sum_{k=1}^n u_k^{(\delta)}-1$,
\begin{equation}\label{5.aux2}
\begin{aligned}
  \big\|\varphi_i^L(u^{(\delta)})-\varphi_i^K(u^{(\delta)})\big\|_{L^1(Q_T)}
	&\le \big\|u_i^{(\delta)}\big(\varphi(v_\delta^L)-\varphi(v_\delta^K)\big)
	\big\|_{L^1(Q_T)} \\
	&\phantom{xx}{}
	+ 2\big\|L\big(1-\varphi(v_\delta^L)\big) 
	- K\big(1-\varphi(v_\delta^K)\big)\|_{L^1(Q_T)}\\
	&=: I_1 + I_2.
\end{aligned}
\end{equation}
By the mean value theorem, H\"older's inequality, and the uniform bounds
in Lemma \ref{lem.delta}, we find that
\begin{align*}
  I_1 &\le \max_{s\in\R}|\varphi'(s)|\bigg\|u_i^{(\delta)}
	\bigg(\frac{1}{L}\sum_{k=1}^n u_k^{(\delta)}-\frac{1}{K}\sum_{k=1}^n u_k^{(\delta)}
	\bigg)\bigg\|_{L^1(Q_T)} \\
	&\le \max_{s\in\R}|\varphi'(s)|\bigg(\frac{1}{L}-\frac{1}{K}\bigg)
	\sum_{k=1}^n\|u_i^{(\delta)}u_k^{(\delta)}\|_{L^1(Q_T)} \\
	&\le \max_{s\in\R}|\varphi'(s)|\frac{1}{L}\|u_i^{(\delta)}\|_{L^2(Q_T)}
	\sum_{k=1}^n\|u_k^{(\delta)}\|_{L^2(Q_T)} 
	\le \frac{C(u^0,b,T)}{L}.
\end{align*}
For the expression $I_2$, we use the property $\varphi(-1)=1$ and the mean value
theorem again:
\begin{align*}
  I_2 &= 2\big\|L\big(\varphi(-1)-\varphi(v_\delta^L)\big)
	- K\big(\varphi(-1)-\varphi(v_\delta^K)\big)\big\|_{L^1(Q_T)} \\
	&= 2\bigg\|\sum_{k=1}^n u_k^{(\delta)}\bigg\{
	\varphi'\bigg(-1+\frac{\theta_1}{L}\sum_{k=1}^n u_k^{(\delta)}\bigg)
	- \varphi'\bigg(-1+\frac{\theta_2}{K}\sum_{k=1}^n u_k^{(\delta)}\bigg)\bigg\}
	\bigg\|_{L^1(Q_T)} \\
	&= 2\bigg\|\bigg(\sum_{k=1}^n u_k^{(\delta)}\bigg)^2
	\varphi''(\xi)\bigg(\frac{\theta_1}{L}-\frac{\theta_2}{K}\bigg)\bigg\|_{L^1(Q_T)},
\end{align*}
where $\theta_1=\theta_1(x,t)$, $\theta_2=\theta_2(x,t)\in(0,1)$ and $\xi\in\R$.
Therefore,
$$
  I_2 \le 2\max_{s\in\R}|\varphi''(s)|\frac{1}{L}\bigg\|\sum_{k=1}^n u_k^{(\delta)}
	\bigg\|_{L^1(Q_T)}^2 \le \frac{C(u^0,b,T)}{L}.
$$
We infer from \eqref{5.aux2} that
$$
  \big\|\varphi_i^L(u^{(\delta)})-\varphi_i^K(u^{(\delta)})\big\|_{L^1(Q_T)}
	\le \frac{C(u^0,b,T)}{L}.
$$

This estimate and the convergence \eqref{5.conv} allow us to conclude that
$(v_i^L)$ is a Cauchy sequence. Indeed, we find that
\begin{align*}
  \|v_i^L-v_i^K\|_{L^1(Q_T)}
	&\le \big\|v_i^L-\varphi_i^L(u^{(\delta)})\big\|_{L^1(Q_T)}
	+ \big\|\varphi_i^L(u^{(\delta)})-\varphi_i^K(u^{(\delta)})\big\|_{L^1(Q_T)} \\
	&\phantom{xx}{}+ \big\|\varphi_i^K(u^{(\delta)})-v_i^K\big\|_{L^1(Q_T)} \\
	&\le \big\|v_i^L-\varphi_i^L(u^{(\delta)})\big\|_{L^1(Q_T)}
	+ \frac{C(u^0,b,T)}{L}
	+  \big\|\varphi_i^K(u^{(\delta)})-v_i^K\big\|_{L^1(Q_T)}.
\end{align*}
In the limit $\delta\to 0$, we infer from \eqref{5.conv} that for all
$K$, $L\in\N$ with $K>L$,
$$
  \|v_i^L-v_i^K\|_{L^1(Q_T)}\le \frac{C(u^0,b,T)}{L},
$$
which proves the claim. Consequently, there exist functions $u_i\in L^1(Q_T)$ with
$u_i\ge 0$ in $Q_T$ such that $v_i^L\to u_i$ strongly in $L^1(Q_T)$ as $L\to\infty$.

Next, we prove that $u_i^{(\delta)}\to u_i$ a.e. For this, we proceed similarly as
in \cite[p.~572]{Fis15} and show that $(u_i^{(\delta)})$ converges in measure. Since
$$
  |u_i^{(\delta)}-u_i| \le \big|u_i^{(\delta)}-\varphi_i^L(u^{(\delta)})\big|
	+ \big|\varphi_i^L(u^{(\delta)})-v_i^L\big| + |v_i^L-u_i|,
$$
we have for any $\eps>0$,
\begin{equation}\label{5.aux3}
\begin{aligned}
  \mbox{meas}&\big(\big\{(x,t)\in Q_T:|u_i^{(\delta)}(x,t)-u_i(x,t)|>\eps\big\}\big) 
	\le \mbox{meas}(\{u_i^{(\delta)}\neq \varphi_i^L(u^{(\delta)})\}) \\
	&\phantom{xx}{}
	+ \mbox{meas}\bigg(\bigg\{|\varphi_i^L(u^{(\delta)})-v_i^L|>\frac{\eps}{2}
	\bigg\}\bigg) 
	+ \mbox{meas}\bigg(\bigg\{|v_i^L-u_i|>\frac{\eps}{2}\bigg\}\bigg).
\end{aligned}
\end{equation}
By (L2), $u_i^{(\delta)}\neq \varphi_i^L(u^{(\delta)})$ implies that
$\sum_{k=1}^n u_i^{(\delta)}\ge L$. Therefore, by the uniform $L^1$ bound for 
$u^{(\delta)}$, the first term on the right-hand side can be estimated as follows:
\begin{align*}
  \mbox{meas}&(\{u_i^{(\delta)}\neq\varphi_i^L(u^{(\delta)})\})
	\le \mbox{meas}\bigg(\bigg\{\sum_{k=1}^n u_k^{(\delta)}\ge L\bigg\}\bigg) \\
	&= \frac{1}{L}\int_0^T\int_\Omega\chi_{\{\sum_{k=1}^n u_k^{(\delta)}\ge L\}}Ldxdt
	\le \frac{1}{L}\int_0^T\int_\Omega\sum_{k=1}^n u_k^{(\delta)} dxdt
	\le \frac{C(u^0,b,T)}{L},
\end{align*}
where $\chi_A$ is the characteristic function on the set $A$. Similarly,
\begin{align*}
  \mbox{meas}\bigg(\bigg\{|\varphi_i^L(u^{(\delta)})-v_i^L|>\frac{\eps}{2}
	\bigg\}\bigg) 
	&\le \frac{2}{\eps}\big\|\varphi_i^L(u^{(\delta)})-v_i^L\big\|_{L^1(Q_T)} \\
	\mbox{meas}\bigg(\bigg\{|v_i^L-u_i|>\frac{\eps}{2}\bigg\}\bigg)
	&\le \frac{2}{\eps}\|v_i^L-u_i\|_{L^1(Q_T)}.
\end{align*}
Therefore, \eqref{5.aux3} gives for any $\eps>0$,
$$
  \mbox{meas}(\{|u_i^{(\delta)}-u_i|>\eps\})
	\le \frac{C(u^0,b,T)}{L} + \frac{2}{\eps}\|\varphi_i^L(u^{(\delta)})-v_i^L
	\big\|_{L^1(Q_T)} + \frac{2}{\eps}\|v_i^L-u_i\|_{L^1(Q_T)}.
$$
We infer from \eqref{5.conv} that in the limit $\delta\to 0$,
$$
  \limsup_{\delta\to 0}\mbox{meas}(\{u_i^{(\delta)}-u_i|>\eps\})
	\le \frac{C(u^0,b,T)}{L} + \frac{2}{\eps}\|v_i^L-u_i\|_{L^1(Q_T)}.
$$
As $v_i^L\to u_i$ a.e., the limit $L\to\infty$ then gives
$$
  \lim_{\delta\to 0}\mbox{meas}(\{|u_i^{(\delta)}-u_i|>\eps\}) = 0.
$$
This shows that $(u_i^{(\delta)})$ converges in measure. Hence, there exists
a subsequence, which is not relabeled, such that $u_i^{(\delta)}\to u_i$
a.e.\ in $Q_T$. The uniform bound for $(u_i^{(\delta)})$ in $L^{2+2/d}(Q_T)$
(see Lemma \ref{lem.delta})
implies that $u_i^{(\delta)}\to u_i$ strongly in $L^2(Q_T)$, which proves
\eqref{5.strong}. By the same lemma, also the weak convergence \eqref{5.weak}
follows (again up to a subsequence). Moreover, by Fatou's lemma,
$\int_\Omega h(u)dx<\infty$. 

{\em Step 4.} Next, we verify identity \eqref{5.weak1} by passing to the
limit $\delta\to 0$ in the weak formulation \eqref{5.weak0}.
We observe that, using the mean value theorem and (L5), 
$$
  \|\varphi_i^L(u^{(\delta)})-\varphi_i^L(u)\|_{L^2(Q_T)}
	\le \sup_{v\in(0,\infty)^n}|(\varphi_i^L)'(v)|\|u^{(\delta)}-u\|_{L^2(Q_T)}
	\le C\|u^\delta-u\|_{L^2(Q_T)},
$$
where $C>0$ is here and in the following a constant which is
independent of $\delta$ (and $L$). Consequently, the $L^2$ convergence
of $(u^{(\delta)})$ shows that $\varphi_i^L(u^{(\delta)})\to\varphi_i^L(u)$
strongly in $L^2(Q_T)$ as $\delta\to 0$, 
and the first integral in \eqref{5.weak0} converges:
$$
  \int_0^T\int_\Omega\varphi_i^L(u^{(\delta)})\pa_t\psi dxdt
	\to \int_0^T\int_\Omega\varphi_i^L(u)\pa_t\psi dxdt.
$$
By (L3) and (L5), the sequence 
$(\pa_j\varphi_i^L(u^{(\delta)})A_{j\ell}(u^{(\delta)}))$ is bounded
in $L^\infty(Q_T)$ with respect to $\delta$. 
We conclude from the convergence \eqref{5.strong} that
$\pa_j\varphi_i^L(u^{(\delta)})A_{j\ell}(u^{(\delta)})\to
\pa_j\varphi_i^L(u)A_{j\ell}(u)$ strongly in $L^2(Q_T)$. Together with
the weak convergence \eqref{5.weak} of the gradients, we infer that
\begin{align*}
  \int_0^T\int_\Omega\pa_j\varphi_i^L(u^{(\delta)})&\bigg(\sum_{\ell=1}^n
	A_{j\ell}(u^{(\delta)})\na u_\ell^{(\delta)}\bigg)\cdot\na\psi dxdt \\
	&\to \int_0^T\int_\Omega\pa_j\varphi_i^L(u)\bigg(\sum_{\ell=1}^n A_{j\ell}(u)
	\na u_\ell\bigg)\cdot\na\psi dxdt.
\end{align*}

Again using (L3), we have
$$
  \big\|\pa_j\varphi_i^L(u^{(\delta)})u_j^{(\delta)}\big\|_{L^\infty(Q_T)}
	+ \bigg\|\pa_j\varphi_i^L(u^{(\delta)})\frac{f_j(u^{(\delta)})}{1+\delta
	|f(u^{(\delta)})|}\bigg\|_{L^\infty(Q_T)} \le C(L).
$$
Consequently, by \eqref{5.strong},
$\pa_j\varphi_i^L(u^{(\delta)})u_j^{(\delta)}\to \pa_j\varphi_i^L(u)u_j$
and 
$$
  \pa_j\varphi_i^L(u^{(\delta)})\frac{f_j(u^{(\delta)})}{1+\delta|f(u^{(\delta)})|}
	\to \pa_j\varphi_i^L(u)f_j(u)\quad\mbox{strongly in }L^2(Q_T).
$$
This allows us to perform the limit $\delta\to 0$ in
the drift and reaction terms:
\begin{align*}
  \int_0^T\int_\Omega\sum_{j=1}^n\pa_j\varphi_i^L(u^{(\delta)})u_j^{(\delta)}
	b_j\cdot\na\psi dxdt
	&\to \int_0^T\int_\Omega\sum_{j=1}^n\pa_j\varphi_i^L(u)u_jb_j\cdot\na\psi dxdt \\
	\int_0^T\int_\Omega\sum_{j=1}^n\pa_j\varphi_i^L(u^{(\delta)})
	\frac{f_j(u^{(\delta)})\psi}{1+\delta|f(u^{(\delta)})|}dxdt
	&\to \int_0^T\int_\Omega\sum_{j=1}^n\pa_j\varphi_i^L(u)f_j(u)\psi dxdt.
\end{align*}

It remains to perform the limit $\delta\to 0$ in the integral involving
the second derivatives of $\varphi_i^L(u^{(\delta)})$ in \eqref{5.weak0}.
Define the signed Radon measures
$$
  \mu_i^{(\delta,L)} := \sum_{j,k=1}^n\pa_j\pa_k\varphi_i^L(u^{(\delta)})
	\bigg(\sum_{\ell=1}^n A_{j\ell}(u^{(\delta)})\na u_\ell^{(\delta)}
	- u_j^{(\delta)}b_j\bigg)\cdot\na u_k^{(\delta)}dxdt.
$$
It follows from (L7), the Cauchy-Schwarz inequality, and
the uniform bounds in Lemma \ref{lem.delta} that
\begin{align*}
  |\mu_i^{(\delta,L)}|(Q_T)
	&\le C\int_0^T\int_\Omega\sum_{k=1}^n\big(|\na u^{(\delta)}|
	+ \|b\|_{L^\infty(Q_T)}\big)|\na u_k^{(\delta)}|dxdt \\
	&\le C\big(\|b\|_{L^\infty(Q_T)}^2 + \|\na u^{(\delta)}\|_{L^2(Q_T)}^2\big)
	\le C(u^0,b,T).
\end{align*}
The weak-star compactness criterium for Radon measures \cite[Corollary 4.34]{Mag12}
then implies the existence of a subsequence (not relabeled) such that
$\mu_i^{(\delta,L)}\rightharpoonup^* \mu_i^L$ as $\delta\to 0$
in the sense of measures, where $\mu_i^L$ is a signed Radon measure. 
Hence,
\begin{align*}
  \int_0^T\int_\Omega\sum_{j,k=1}^n & \pa_j\pa_k\varphi_i^L(u^{(\delta)})
	\bigg(\sum_{\ell=1}^n A_{j\ell}(u^{(\delta)})\na u_\ell^{(\delta)}
	- u_j^{(\delta)}b_j\bigg)\cdot\na u_k^{(\delta)}\psi dxdt \\
	&\to \int_0^T\int_\Omega \psi d\mu_i^L(x,t).
\end{align*}
This shows \eqref{5.weak1}.

{\em Step 5.} The final step is the proof of the convergence \eqref{5.mu}.
For this, we write
\begin{align*}
  |\mu_i^{(\delta,L)}|(Q_T)
	&\le \sum_{j,k,\ell=1}^n\int_0^T\int_\Omega|\pa_j\pa_k\varphi_i^L(u^{(\delta)})| \\
	&\phantom{xx}{}\times\big(|A_{j\ell}(u^{(\delta)})|\,|\na u_\ell^{(\delta)}|
	+ |u_j^{(\delta)}|\|b\|_{L^\infty(Q_T)}\big)|\na u_k^{(\delta)}| dxdt \\
	&= \sum_{j,k,\ell=1}^n\sum_{K=1}^\infty\int_0^T\int_\Omega
	\chi_{\{K-1\le|u^{(\delta)}|<K\}}|\pa_j\pa_k\varphi_i^L(u^{(\delta)})| \\
	&\phantom{xx}{}\times\big(|A_{j\ell}(u^{(\delta)})|\,|\na u_\ell^{(\delta)}|
	+ |u_j^{(\delta)}|\|b\|_{L^\infty(Q_T)}\big)|\na u_k^{(\delta)}| dxdt \\
	&\le C\sum_{K=1}^\infty F_K^{(\delta)}G_K^L,
\end{align*}
where
\begin{align*}
	F_K^{(\delta)} &= \int_0^T\int_\Omega \chi_{\{K-1\le|u^{(\delta)}|<K\}}
	\big(|\na u^{(\delta)}| + \|b\|_{L^\infty(Q_T)}\big)|\na u^{(\delta)}|dxdt, \\
  G_K^L &= \sum_{j,k=1}^n\sup_{K-1\le|v|<K}(1+|v|)|\pa_j\pa_k\varphi_i^L(v)|.
\end{align*}
For the last inequality, we have used the fact that 
$A_{j\ell}(u^{(\delta)})$ depends linearly on $u_i^{(\delta)}$, i.e.\
$A_{j\ell}(u^{(\delta)})\le C(1+|u^{(\delta)}|)$.
By (L3), for any fixed $L\in\N$ and sufficiently large $K\in\N$, it holds that
$\pa_j\pa_k\varphi_i^L(v)=0$ for all $K-1\le|v|<K$ and consequently $G_K^L=0$.
This means that for any fixed $L\in\N$, the sum $\sum_{K=1}^\infty F_K^{(\delta)}G_L^K$
contains only a finite number of nonvanishing terms. 
By the weak-star lower semicontinuity
of the total variation of signed Radon measures on open sets
\cite[Prop.~4.29]{Mag12}, we deduce from $\mu_i^{(\delta,L)}\rightharpoonup^* \mu_i^L$
as $\delta\to 0$ that for some subsequence,
\begin{equation}\label{5.sub}
\begin{aligned}
  |\mu_i^L|(\overline\Omega\times[0,T))
	&= |\mu_i^{L}|(Q_T)
	\le\liminf_{\delta\to 0}|\mu_i^{(\delta,L)}|(Q_T) \\
	&\le C\liminf_{\delta\to 0}\sum_{K=1}^\infty F_K^{(\delta)} G_K^L 
	= C\sum_{K=1}^\infty G_K^L\lim_{\delta\to 0}F_K^{(\delta)}
\end{aligned}
\end{equation}
In the last equality, we have selected a common subsequence such that 
$(F_K^{(\delta)})$ converges.
It follows from the Cauch-Schwarz inequality and the uniform bounds in Lemma
\ref{lem.delta} that
\begin{align*}
  \sum_{K=1}^\infty F_K^{(\delta)} 
	&= \int_0^T\int_\Omega\big(|\na u^{(\delta)}| + \|b\|_{L^\infty(Q_T)}\big)
	|\na u^{(\delta)}| dxdt \\
	&\le C\big(\|\na u^{(\delta)}\|_{L^2(Q_T)}^2 + \|b\|_{L^\infty(Q_T)}^2\big)
	\le C(u^0,b,T).
\end{align*}
By Fatou's lemma, we have for the same subsequence as in \eqref{5.sub},
$$
  \sum_{K=1}^\infty\lim_{\delta\to 0}F_K^{(\delta)}
	\le \liminf_{\delta\to 0}\sum_{K=1}^\infty F_K^{(\delta)} \le C(u^0,b,T).
$$
It follows from (L7) that $0\le G_K^L\le C$ for some constant $C>0$, and hence
$$
  0\le G_K^L\lim_{\delta\to 0}F_K^{(\delta)}\le C\lim_{\delta\to 0}F_K^{(\delta)}.
$$
Moreover, we infer from (L6) that 
$$
  0\le \lim_{L\to\infty}G_K^L \le (1+K)\sum_{j,k=1}^n\lim_{L\to\infty}
	\sup_{v\in[0,\infty)^n}|\pa_j\pa_k\varphi_i^L(v)| = 0,
$$
i.e.\ $\lim_{L\to\infty}G_K^L=0$. We deduce from Lebesgue's dominated convergence theorem
that
$$
  0\le\lim_{L\to\infty}|\mu_i^L|(\overline\Omega\times[0,T)) 
	\le C\lim_{L\to\infty}\sum_{K=1}^\infty
	G_K^L\lim_{\delta\to 0}F_K^{(\delta)} 
	= C\sum_{K=1}^\infty\lim_{L\to\infty}G_K^L\lim_{\delta\to 0}F_K^{(\delta)} = 0.
$$
This ends the proof of \eqref{5.mu}.
\end{proof}

\begin{proof}[Proof of Theorem \ref{thm.ex}]
We apply Lemma \ref{lem.chain} with $v_i=\varphi_i^L(u)\in L^2(0,T;H^1(\Omega))$,
$v_{0,i}=\varphi_i^L(u^0)\in L^1(\Omega)$, $w_i=\sum_{j=1}^n\pa_j\varphi_i^L(u)f_j(u)
\in L^1(Q_T)$, $\mu_i=-\mu_i^L\in{\mathcal M}(\overline\Omega\times[0,T))$, 
$q_i=0$, and
$$
  z_i = -\sum_{j=1}^n\pa_j\varphi_i^L(u)\bigg(\sum_{\ell=1}^n A_{j\ell}(u)\na u_\ell
	- u_jb_j\bigg)\in L^2(Q_T;\R^n).
$$

Then we obtain from \eqref{5.weak1} that for all $\xi\in C^\infty([0,\infty)^n)$ with
$\xi'\in C_0^\infty([0,\infty);\R^n)$ and for all $\phi\in C_0^\infty(\overline\Omega
\times[0,T))$,
\begin{align}
  \bigg|-&\int_0^T\int_\Omega\xi(\varphi^L(u))\pa_t\phi dxdt
	- \int_\Omega \xi(\varphi^L(u^0))\phi(\cdot,0)dx \nonumber \\
	&\phantom{xx}{}+ \int_0^T\int_\Omega\sum_{i,k=1}^n\pa_i\pa_k\xi(\varphi^L(u))
	\sum_{j=1}^n\pa_j\varphi_i^L(u)\bigg(\sum_{\ell=1}^n A_{j\ell}(u)\na u_\ell
	- u_jb_j\bigg) \nonumber \\
	&\phantom{xxxx}{}\times\sum_{m=1}^n\pa_m\varphi_k^L(u)\na u_m \phi dxdt 
	\label{5.L} \\
	&\phantom{xx}{}+ \int_0^T\int_\Omega\sum_{i=1}^n\pa_i\xi(\varphi^L(u))
	\sum_{j=1}^n\pa_j\varphi_i^L(u)\bigg(\sum_{\ell=1}^n A_{j\ell}(u)\na u_\ell
	- u_jb_j\bigg)\cdot\na\phi dxdt \nonumber \\
	&\phantom{xx}{}- \int_0^T\int_\Omega\sum_{i=1}^n\pa_i\xi(\varphi^L(u))\sum_{j=1}^n
	\pa_j\varphi_i^L(u)f_j(u)\phi dxdt\bigg| \nonumber \\
	&\le C(\Omega)\|\phi\|_{L^\infty(Q_T)}\sup_{v\in[0,\infty)^n}|\xi'(v)|
	\sum_{i=1}^n|\mu_i^L|(\overline\Omega\times[0,T)). \nonumber
\end{align}

We wish to perform the limit $L\to\infty$ in \eqref{5.L}. 
We deduce from the mean value theorem and (L3) that 
$$
   \|\xi(\varphi^L(u))-\xi(u)\|_{L^1(Q_T)}
	 \le\sup_{v\in[0,\infty)^n}|\xi'(v)|\,\|\varphi^L(u)-u\|_{L^1(Q_T)}
	\le C\|\varphi^L(u)-u\|_{L^1(Q_T)}.
$$
By definition of $\varphi_i^L$, $\varphi_i^L(u)$ converges pointwise to $u_i$
as $L\to\infty$. Together with the linear bound for $\varphi_i^L$ from (L1),
$0\le\varphi_i^L(u)\le u_i+2\sum_{k=1}^n u_k\in L^1(Q_T)$, which gives a uniform
bound, we can apply LebesgueÄs dominated convergence theorem to conclude that
$\varphi_i^L(u)\to u_i$ strongly in $L^1(Q_T)$ and consequently,
$\xi(\varphi^L(u))\to \xi(u)$ strongly in $L^1(Q_T)$. Similarly,
$\xi(\varphi^L(u^0))\to\xi(u^0)$ strongly in $L^1(Q_T)$. Therefore, the first
two integrals in \eqref{5.L} converge:
\begin{align*}
  -\int_0^T\int_\Omega\xi(\varphi^L(u))\pa_t\phi dxdt 
	&\to -\int_0^T\int_\Omega\xi(u)\pa_t\phi dxdt, \\
	-\int_\Omega\xi(\varphi^L(u^0))\phi(\cdot,0)dx
	&\to -\int_\Omega\xi(u^0)\phi(\cdot,0) dx.
\end{align*}

Next, consider the last integral on the left-hand side of \eqref{5.L}.
Let $L_0>0$ be such that $\operatorname{supp}\xi'\subset[0,L_0/n)^n$ and let $L>L_0$.
We distinguish the cases $\sum_{i=1}^n u_i\ge L_0$ and
$\sum_{i=1}^n u_i < L_0$. In the former case, it follows from (L8) 
that $\sum_{i=1}^n\varphi_i^L(u)\ge L_0$ or $\varphi^L(u)\not\in
[0,L_0/n)^n$ and, in particular, $\varphi^L(u)\not\in\operatorname{supp}\xi'$.
Hence, $\pa_i\xi(\varphi^L(u))=0$ and $\pa_i\pa_k\xi(\varphi^L(u))=0$.
In the latter case, we deduce from (L2) that $\varphi^L(u)=u$ and consequently
$\pa_j\varphi_i^L(u)=\delta_{ij}$. Furthermore, we have
$$
  \{u\in\operatorname{supp}\xi'\}
	\subset \bigg\{u\in \bigg[0,\frac{L_0}{n}\bigg)^n\bigg\}
	\subset\bigg\{\sum_{i=1}^n u_i<L_0\bigg\}.
$$
This allows us to reformulate the last term on the left-hand side of \eqref{5.L}:
\begin{align*}
  \int_0^T\int_\Omega & \sum_{i=1}^n\pa_i\xi(\varphi^L(u))\sum_{j=1}^n
	\pa_j\varphi_i^L(u)f_j(u)\phi dxdt \\
	&= \int_0^T\int_\Omega\sum_{i=1}^n\chi_{\{\sum_{i=1}^n u_i\ge L_0\}}
	\pa_i\xi(\varphi^L(u))\sum_{j=1}^n\pa_j\varphi_i^L(u)f_j(u)\phi dxdt \\
	&\phantom{xx}{}+ \int_0^T\int_\Omega\sum_{i=1}^n\chi_{\{\sum_{i=1}^n u_i<L_0\}}
	\pa_i\xi(\varphi^L(u))\sum_{j=1}^n\pa_j\varphi_i^L(u)f_j(u)\phi dxdt \\
	&= \int_0^T\int_\Omega\sum_{i=1}^n\chi_{\{\sum_{i=1}^n u_i<L_0\}}
	\pa_i\xi(u)f_i(u)\phi dxdt \\
	&= \int_0^T\int_\Omega\sum_{i=1}^n\pa_i\xi(u)f_i(u)\phi dxdt,
\end{align*}
and this expression does not depend on $L$. In a similar way, we compute
the third and fourth integrals on the left-hand side of \eqref{5.L}:
\begin{align*}
  \int_0^T\int_\Omega&\sum_{i,k=1}^n\pa_i\pa_k\xi(\varphi^L(u))
	\sum_{j=1}^n\pa_j\varphi_i^L(u)\bigg(\sum_{\ell=1}^n A_{j\ell}(u)\na u_\ell
	- u_jb_j\bigg) \\
	&\phantom{xx}{}\times\sum_{m=1}^n\pa_m\varphi_k^L(u)\na u_m \phi dxdt \\
	&= \int_0^T\int_\Omega\sum_{i,k=1}^n\pa_i\pa_k\xi(u)\bigg(
	\sum_{\ell=1}^n A_{j\ell}(u)\na u_\ell - u_jb_j\bigg)\cdot\na u_k \phi dxdt, \\
	\int_0^T\int_\Omega&\sum_{i=1}^n\pa_i\xi(\varphi^L(u))
	\sum_{j=1}^n\pa_j\varphi_i^L(u)\bigg(\sum_{\ell=1}^n A_{j\ell}(u)\na u_\ell
	- u_jb_j\bigg)\cdot\na\phi dxdt \\
	&= \int_0^T\int_\Omega\sum_{i=1}^n\pa_i\xi(u)
	\bigg(\sum_{\ell=1}^n A_{j\ell}(u)\na u_\ell - u_jb_j\bigg)\cdot\na\phi dxdt.
\end{align*}
Finally, because of \eqref{5.mu}, the right-hand side of \eqref{5.L} vanishes
in the limit $L\to\infty$. Therefore, passing to the limit $L\to\infty$ in
\eqref{5.L}, we see that \eqref{1.renorm} holds. This concludes the proof.
\end{proof}

%%%%%%%%%%%%%%%%%%%%%%%%%%%%%%%%%%%%%%%%%%%%%%%%%%%%%%%%%%%%%%%%%%%%%%%%%%%%

\end{document}